\documentclass[reqno,12pt,a4paper]{amsart}

\usepackage[margin=2cm]{geometry}

\usepackage{amsmath,amsthm,amssymb,amsfonts,dsfont,ifpdf,mathtools,verbatim}
\usepackage{enumitem}
\usepackage{graphicx}
\usepackage[usenames,dvipsnames]{color}
\usepackage[all]{xy}
\usepackage{mathrsfs}

\usepackage{csquotes}
\MakeOuterQuote{"}

\usepackage[colorlinks=true, linkcolor=blue, urlcolor=blue, citecolor=ForestGreen]{hyperref}

\numberwithin{equation}{section}

\newtheorem {thm}    {Theorem}[section]
\newtheorem {lem}      [thm]    {Lemma}

\newtheorem {cor}  [thm]    {Corollary}
\newtheorem {prop}[thm]    {Proposition}
\newtheorem* {prop*} {Proposition}

\newtheorem*{claim*}   {Claim}

\newtheorem*{conj*} {Conjecture}

\theoremstyle{definition}

\newtheorem {ex}    [thm]    {Example}
\newtheorem {rmk}    [thm]    {Remark}
\newtheorem*{rmk*}  {Remark}

\newtheorem {qst}   [thm]    {Question}
\newtheorem*{qst*} {Question}

\newtheorem* {problem*}{Problem}

\newcounter{AbcT}

\numberwithin{equation}{section}

\newcommand {\supl}   {\sup\limits}

\newcommand {\C} {{\mathbb C}}

\newcommand {\Hyp} {{\mathbb H}}

\newcommand {\N} {{\mathbb N}}
\newcommand {\Q} {{\mathbb Q}}

\newcommand {\R} {{\mathbb R}}
\newcommand {\T} {{\mathbb T}}

\newcommand {\Z} {{\mathbb Z}}

\newcommand {\cB} {{\mathcal B}}

\newcommand {\cE} {{\mathcal E}}
\newcommand {\cF} {{\mathcal F}}

\newcommand {\cM} {{\mathcal M}}

\DeclareMathOperator{\SL}{SL}
\DeclareMathOperator{\PSL}{PSL}
\DeclareMathOperator{\GL}{GL}

\DeclareMathOperator{\Or}{O}
\DeclareMathOperator{\SO}{SO}

\newcommand{\eps}{\varepsilon}

\newcommand {\IGNORE}[1]  {}

\newcommand{\norm}[1]{\left\lVert#1\right\rVert}

\newcommand {\bsl} {\backslash}

\newcommand {\La} {{\Lambda}}
\newcommand{\Ga}{\Gamma}

\renewcommand{\sl}{\mathfrak{sl}}

\newcommand\vol{\operatorname{vol}}

\newcommand\cA{\mathcal{A}}

\newcommand\Ad{\operatorname{Ad}}

\renewcommand{\norm}[1]{\left\lVert#1\right\rVert}
\newcommand{\lnorm}[1]{\bigl\lVert#1\bigr\rVert}

\begin{document}
	
	\title[The distribution of dilating sets]{The distribution of dilating sets: a journey from Euclidean to hyperbolic geometry}
	\author[E.Corso]{Emilio Corso}
	\address{Department of Mathematics, University of British Columbia, 1984 Mathematics Road, Vancouver, BC, V6T 1Z2, Canada}
	\email{corsoemilio2@gmail.com}
	\date{\today}
	
	
	\begin{abstract}
	We survey the distributional properties of progressively dilating sets under projection by covering maps, focusing on manifolds of constant sectional curvature. In the Euclidean case, we review previously known results and formulate some generalizations, derived as a direct byproduct of recent developments on the problem of Fourier decay of finite measures. In the hyperbolic setting, we consider a natural upgrade of the problem to unit tangent bundles; confining ourselves to compact hyperbolic surfaces, we discuss an extension of our recent result with Ravotti on expanding circle arcs, establishing a precise asymptotic expansion for averages along expanding translates of homogeneous curves.
	\end{abstract}
	\maketitle
	
	\tableofcontents
	
	\section{Introduction}

	 What is the ergodic behaviour of progressively dilating sets when the ambient space in which they live is folded according to some prescribed procedure? To the best of our knowledge, the question was first asked by Dennis Sullivan in the early eighties\footnote{ As regards the historical claim made here, we refer to the acknowledgments in Randol's original article~\cite{Randol} on the problem.} of the last century, though partial answers were already provided earlier as a byproduct of the investigation of seemingly unrelated problems, a point we shall clarify in due course. 
	  In this pre-eminently expository article, we intend to survey various contributions to the opening question, an aspect of which we recently investigated in joint work~\cite{Corso-Ravotti} with Ravotti; when appropriate, we shall place emphasis on connections with related, parallel developments about problems pertaining to other domains of mathematical research.

	  Our first order of business is to formulate our original question in a mathematically rigorous fashion. To this end, we appeal to the general setup outlined by Randol in~\cite{Randol}. Let us consider
  a compact connected Riemannian manifold $(M,g)$, whose Riemannian volume measure (see~\cite[Chap.~XXIII]{Lang}) we indicate with $m_M$ or $\vol_M$. Suppose we are given a universal Riemannian covering space $(\tilde{M},\tilde{g})$ of $(M,g)$, with $\pi\colon \tilde{M}\to M$ denoting the covering map. Let $(h_t)_{t\in \R_{>0}}$ be a family of homotheties\footnote{A homothety of a Riemannian manifold $(N,\bar{g})$ is a smooth diffeomorphism $F\colon N\to N$ such that $F_{\ast}\bar{g}=\lambda^{-2} \bar{g}$ for some $\lambda \in \R_{>0}$, where $F_{\ast}\bar{g}$ denotes the pushforward of the Riemannian metric $g$ under $F$; the scalar $\lambda$ is called the ratio of the homothety $h$.} of $\tilde{M}$ whose ratio tends to infinity as $t$ does. Given a Borel probability measure $\mu$ on the covering manifold $\tilde{M}$, let $\mu_t$ be the pushforward of $\mu$ under the homothety $h_t$, and let $m_t$ be the pushforward of $\mu_t$ under the covering map $\pi$, for every $t\in \R_{>0}$.
	
	\begin{qst}
		\label{qst:initial}
	Is it possible to describe the set of limits, in the weak$^*$ topology, of the measures $m_t$ as $t$ tends to infinity?
	\end{qst}

To relate this phrasing to the pictorial formulation we started with, the manifold $(M,g)$ is thought of as the space resulting after a certain folding procedure is operated on a univeral cover $(\tilde{M},\tilde{g})$. Such an operation is described mathematically by the isometric action on $\tilde{M}$ of the fundamental group $\pi_1(M)$ by deck transformations; for concrete examples, see the discussions at the beginning of Sections~\ref{sec:euclidean} and~\ref{sec:hyperbolic}, setting the stage, respectively, for flat and hyperbolic manifolds. The progressive dilations which a given subset of $\tilde{M}$ undergoes are dictated by the collection $(h_t)_{t\in \R_{>0}}$. Finally, the ergodic behaviour of large dilates is embodied by the way they distribute when projected down to the space $M$; in ergodic theory, this is routinely described in measure-theoretic terms, which entails endowing the initial set with a probability measure, represented by $\mu$ in the framework presented above, and then invoking the concept of convergence of measures in the weak$^*$ topology to encode the limiting distributional behaviour of the projected dilates $m_t=(\pi\circ h_t)_*\mu$. 

	\medskip
	Somewhat less pretentiously with respect to Question~\ref{qst:initial}, the present manuscript concerns itself with the problem of identifying sufficient geometric conditions on the initial measure $\mu$ for the set of weak$^*$ limits of $(m_t)_{t>0}$ to consist only of the volume measure $\vol_M$, in which case we say, adhering to a well-established terminology in ergodic theory, that $(m_t)_{t>0}$ equidistributes towards $\vol_M$. All results contained herein pertain exclusively to the case of $(M,g)$ having constant sectional curvature, which already displays considerable challenges and offers a wealth of geometrically relevant examples. 
	
	Upon rescaling the Riemannian metric $g$, which doesn't alter the nature of our problem, there are only three possibilities, up to Riemannian isomorphisms, for the universal covering space $(\tilde{M},\tilde{g})$, according to whether the sectional curvature of $(M,g)$ vanishes, is positive or negative: $(\tilde{M},\tilde{g})$ is then isometric, respectively, to the Euclidean space $\R^{d}$ ($d\geq 1$), the unit sphere $\mathbb{S}^{d}\subset \R^{d+1}$ ($d\geq 2$) equipped with the usual round metric, or the hyperbolic space $\mathbb{H}^{d}$ ($d\geq 2$). This is the well-known Killing-Hopf theorem on the classification of connected, simply connected, complete Riemannian manifolds of constant sectional curvature, stated e.g.~in Bridson-Haefliger's book~\cite[Thm.~3.32]{Bridson-Haefliger}.
	
	\medskip
	We first observe that curvature poses obstructions to the existence of a family of homotheties $(h_t)_{t>0}$ with arbitrarily large ratios. Specifically, every homothety on the sphere $\mathbb{S}^{d}$ is an isometry, that is, it has unit ratio, a fact for which we supply a brief explanation. 
	
	Suppose $h\colon \mathbb{S}^d\to \mathbb{S}^{d}$ is a homothety of ratio $\lambda$; upon replacing $h$ by its inverse, we may assume that $\lambda<1$. Let $L(\gamma)$ denote the length of a continuous, piecewise differentiable curve $\gamma\colon I\to \mathbb{S}^d$ defined over some compact interval $I\subset \R$. It is plain that $L(h\circ \gamma)=\lambda L(\gamma)$ for every such curve $\gamma$, from which it follows that $h$ is a contraction for the Riemannian distance function on $\mathbb{S}^d$ induced by the round metric. By the Banach fixed point theorem, $h$ has a unique fixed point $x_0\in \mathbb{S}^{d}$. Now let $\gamma\colon \R \to \mathbb{S}^{d}$ be a unit-speed geodesic with $\gamma(0)=x_0$. Since $h$ rescales the metric by $\lambda^{-2}$, the curve $t\mapsto h\circ \gamma(\lambda t)$ is also a unit-speed geodesic, passing through $x_0$ at time $t=0$. Both geodesics describe great circles on $\mathbb{S}^d$, which are all of length $2\pi$. In particular, $\gamma(2\pi)=x_0$ and $\gamma(t)\neq x_0$ for every $0<t<2\pi$. On the other hand, $h\circ \gamma(2\pi \lambda)=x_0$; since $h$ is one-to-one, this forces $\gamma(2\pi\lambda)=x_0$, which leads to a contradiction, as $0<2\pi \lambda<2\pi$.

\medskip
	In view of such considerations, our focus will be directed towards the two remaining cases of vanishing and negative curvature\footnote{It is geometrically obvious that in both cases there exist homotheties of arbitrarily large ratio; we shall treat concrete examples of those in the intervening sections.}, that is, when $\tilde{M}=\R^d$ and $\tilde{M}=\Hyp^{d}$. The paper is arranged as follows: Sections~\ref{sec:euclidean} and~\ref{sec:quantitativeeuclidean} are devoted to the Euclidean setting, while Sections~\ref{sec:hyperbolic},~\ref{sec:unittangentbundles} and~\ref{sec:proof} deal with the hyperbolic case. We will be interested in the qualitative as well as in the quantitative aspects of our equidistribution problem; this is to say that, whenever equidistribution is known to occur, we shall always lay stress on explicit estimates for the rate at which it occurs. 
	As far as the Euclidean case is concerned, the main results stated in this article, namely Theorems~\ref{thm:qualitativeequidtorus} and~\ref{thm:quantitativetorus}, amount essentially to a slightly more general reformulation of those established by Randol~\cite{Randol} and Strichartz~\cite{Strichartz}, and appeal solely to classical Fourier analysis on Euclidean spaces. For hyperbolic manifolds, we shall first review Randol's contributions in loc.~cit., and subsequently proceed with a recasting of the problem in terms of geodesic flows acting on unit tangent bundles. In recent joint work with Ravotti~\cite{Corso-Ravotti}, we established, via a spectral method originating in Ratner's work~\cite{Ratner} on effective mixing of goedesic and horocycle flows, a precise asymptotic expansion for averages of smooth functions along expanding circle arcs in compact hyperbolic surfaces. Theorem~\ref{thm:expandingtranslates} presents an extension to geodesic translates of all homogeneous arcs for which equidistribution is expected to occur, namely those which are not contained in a leaf of the weak stable foliation for the geodesic flow (see the recent book~\cite{Hasselblatt}  by Fisher and Hasselblatt for a broad introduction to invariant foliations for hyperbolic flows). A slightly weaker version of Theorem~\ref{thm:expandingtranslates} follows from the results of Bufetov and Forni in~\cite{Bufetov-Forni}; in comparison to their approach, which invokes a deep classification of invariant distributions for the horocycle flow previously stablished by Flaminio and Forni in~\cite{Flaminio-Forni}, our arguments are of a more elementary nature.
	
	The main body of this article ends with Section~\ref{sec:further}, discussing some directions of further research and listing references to the existing literature on related equidistribution problems and their connections to questions in number theory. Finally, for the sake of comparison to the case of hyperbolic surfaces, we upgrade in Appendix~\ref{secapp:app} the effective equidistribution statement for expanding circles in the two-dimensional torus, a particular instance of the results exposed in Section~\ref{sec:quantitativeeuclidean}, to an analogous statement for their canonical lifts to the unit tangent bundle. 
	
	\subsection*{Notational and terminological conventions}
	Henceforth, we omit any specification of the Riemannian metric when speaking about a Riemannian manifold; the metric shall always be easily understood from the context. 
	
	All measures appearing in the text are assumed, tacitly in many places, to be Borel measures.
	
	If $(X,\cA),(Y,\cB)$ are measurable spaces, $f\colon X \to Y$ is a measurable map and $\nu$ is a measure on $(X,\cA)$, we indicate with $f_*\nu$ the pushforward of $\nu$ under $f$, defined by
	\begin{equation*}
		f_*\nu(B)=\nu(f^{-1}(B))\;, \quad B\in \cB.
	\end{equation*}
	
	We recall that a lattice in a locally compact Hausdorff topological group $G$ is a discrete subgroup $\Ga$ of $G$ such that the quotient $\Ga\bsl G$ carries a Radon probability measure which is invariant under the action of $G$ by right translations.
	
	Given two real-valued functions $f,g$ defined on a normed vector space $(E,\norm{\cdot})$  over $\R$, we adopt the classical Landau notation $f=O(g)$ if there exists $C\in \R_{>0}$ and a compact subset $K\subset E$ such that $f(x)\leq Cg(x)$ for every $x\notin K$, while the symbol $o(g)$ stands for a function $f\colon E \to \R$ with the property that, for every $\eps>0$, there exists a compact subset $K\subset E$ such that $f(x)\leq \eps g(x)$ for every $x\notin K$.
	
	\subsection*{Acknowledgments}
	
	We express our gratitude to Giovanni Forni for alerting us to his work~\cite{Bufetov-Forni} with Alexander Bufetov, from which an aforementioned, mildly less general version of Theorem~\ref{thm:expandingtranslates} can already be inferred. Our appreciation is also addressed to the anonymous referee for a careful reading of the manuscript.

	\section{Equidistribution of dilating sets in tori}
\label{sec:euclidean}
	
	Our investigations begin with the zero-curvature case, for which we follow the treatment of Randol~\cite{Randol} and Strichartz~\cite{Strichartz} rather closely. If $M$ is a compact, connected, flat Riemannian manifold, then a Riemannian isometry allows to identify it with a quotient manifold $\R^{d}/\Ga$, where $\Ga$ is a discrete cocompact subgroup of the isometry group $\text{Isom}(\R^d)$ (see, for instance,~\cite[Cor.~2.33]{Lee}). The classical Mazur-Ulam theorem (see, e.g.,~\cite[Thm.~2.20]{Einsiedler-Ward})  allows to identify the latter group with the semidirect product $\Or_d(\R)\ltimes \R^d$, where $\Or_d(\R)=\{A\in \GL_d(\R):{^{t}}A \;A=\mathds{1}_d  \}$ is the orthogonal group in dimension $d$. It follows that $M$ is finitely and isometrically covered\footnote{This fact was first established by Bieberbach~\cite{Bieberbach,Bieberbach-second}.} by a torus $\R^{d}/\Lambda$, $\La$ being a lattice in $\R^d$.  Therefore, no loss of generality occurs by treating the case $M=\R^d/\La$, equipped with the standard flat Riemannian structure. As shall clearly emerge from forthcoming computations (see also Remark~\ref{rmk:generaltorusequid}(b)), it is 
equally harmless to assume that $\La=\Z^{d}$ is the standard integer lattice, so that $M=\T^{d}$ is the standard $d$-dimensional torus. The volume measure on $\T^{d}$ is the Haar probability measure $m_{\T^d}$; the covering map $\pi$ is given by $\pi(x)=x+\Z^d$ for every $x\in \R^{d}$. 

We first perform some reductions. Let $h\colon \R^d\to \R^d$ be a Riemannian homothety of ratio $\lambda$. Then the the assignment $x\mapsto \lambda^{-1}h(x)$ defines a Euclidean isometry $\R^d\to \R^d$; there exist therefore $A\in \Or_d(\R)$ and $b\in \R^d$ such that $h(x)=\lambda A(x)+b$ for every $x\in \R^d$. 

If now $(h_t)_{t>0}$ is a family of homotheties $h_t\colon \R^d\to \R^d$ with ratio tending to infinity with $t$, then we may and shall assume, upon changing the parametrization, that $t$ is the ratio of $h_t$. We may thus write $h_t(x)=tA_t(x)+b_t$ for every $x\in \R^d$ and $t>0$, where $A_t\in \Or_d(\R)$ and $b_t\in \R^d$.  
Recall that, as pointed out in the introduction, $h_t$ has a unique fixed point as soon as $t>1$, which we refer to as the center of $h_t$. It is geometrically natural, for the problem we are interested in, to assume that the maps $h_t$ have a common center. For the sake of notational simplicity, we assume it to be the origin, so that $b_t=0$ for every $t\;$; likewise, we assume that $A_t$ is the identity map for every $t$, and refer to Remarks~\ref{rmk:generaltorusequid}(c) and~\ref{rmk:quantitativeremark}(c) for the straightforward generalization to arbitrary $A_t$ and $b_t$.

\subsection{A characterization of uniform distribution: decay along integral rays}
\label{sec:qualitativetorus}
	  
We have thus reduced matters to the standard family of linear homotheties $h_t(x)=tx$, $x\in \R^{d}$,\linebreak $t\in \R_{>0}$. We fix a Borel probability measure $\mu$ on $\R^{d}$, and let $\mu_t$ and $m_t$ be defined as in the introduction. We shall first be interested in the equidistribution problem from a qualitative standpoint, and set out to pinpoint Fourier-analytic conditions on $\mu$ ensuring that the measures $m_t$ equidistribute towards the uniform measure $m_{\T^d}$. Notice that the use of Fourier analysis in equidistribution problems on tori is exceedingly classical, dating back at least to Weyl's seminal article~\cite{Weyl}.

	Denote by $\mathrm{C}(\T^{d})$ the complex Banach space of continuous functions defined on $\T^{d}$, equipped with the supremum norm. For any function $f\in \mathrm{C}(\T^{d})$, we indicate with $\hat{f}\colon \Z^{d}\to \C$ its Fourier transform, defined as
	\begin{equation*}
		\hat{f}(N)=\int_{\T^{d}}f(y+\Z^d)e^{-2\pi i N\cdot (y+\Z^d)}\text{d}m_{\T^{d}}(y+\Z^d)\;, \quad N\in\Z^{d},
	\end{equation*} 
	where $v\cdot w$ denotes the standard Euclidean inner product\footnote{There is no ambiguity in the notation $e^{2\pi i N\cdot (y+ \Z^d)}$ for $N\in \Z^{d}$ and $y\in \R^{d}$, as $e^{2\pi i m}=1$ for any $m\in \Z$.} of two vectors $v,w\in \R^{d}$. Recall that, by the Riemann-Lebesgue lemma (see~\cite[Thm.~8.22]{Folland}), $\hat{f}$ vanishes at infinity.
	
	Let us fix $t>0$ and a continuous function $f\colon \R^{d}\to \C$; for the purposes of the upcoming computations, it is convenient to assume that $\hat{f}\in \ell^{1}(\Z^d)$, the space of summable functions $\Z^d\to \C$. This is a harmless restriction in view of our aim, for weak$^*$ convergence of probability measures can be tested using any dense subset of $\mathrm{C}(\T^{d})$, and smooth functions certainly fulfill the condition (cf.~\cite[Thm.~8.22]{Folland}).

	 Because of our assumption on $\hat{f}$, the Fourier series 
	\begin{equation*}
		\sum_{N\in \Z^d}\hat{f}(N)e^{2\pi i N\cdot (x+\Z^d)}
	\end{equation*}
	of $f$ is normally convergent in $\mathrm{C}(\T^{d})$, in particular converges uniformly towards $f$. The dominated convergence theorem thus delivers
	\begin{equation*}
		\int_{\T^d}f\;\text{d}m_t=\sum_{N\in \Z^d}\hat{f}(N)\int_{\T^{d}}e^{2\pi i N\cdot (x+\Z^d)}\text{d}m_{t}(x+\Z^d)\;;
	\end{equation*}
	unravelling the definition of the measure $m_t$, we may write, for any $N\in \Z^{d}$,
	\begin{equation*}
		\label{eq:unravel}
		\int_{\T^{d}}e^{2\pi i N\cdot (x+\Z^d)}\text{d}m_{t}(x+\Z^d)=\int_{\R^{d}}e^{2\pi i N\cdot x}\text{d}\mu_{t}(x)=\int_{\R^{d}}e^{2\pi i N\cdot tx}\text{d}\mu(x)=\int_{\R^{d}}e^{2\pi i tN\cdot x}\text{d}\mu(x)\;.
	\end{equation*}
	Recall that the Fourier transform of the measure $\mu$ is defined as the function $\hat{\mu}\colon \R^{d}\to \C$ given by
	\begin{equation*}
		\hat{\mu}(\xi)=\int_{\R^{d}}e^{-2\pi i \xi \cdot y}\text{d}\mu(y)\;,\quad \xi \in \R^{d}.
	\end{equation*}
	We have thus obtained
	\begin{equation}
		\label{eq:integralexpansion}
		\int_{\T^{d}}f\;\text{d}m_t=\sum_{N\in \Z^d}\hat{f}(N)\hat{\mu}(-tN)\;.
	\end{equation}
	Hence, the difference $\int_{\T^{d}}f\;\text{d}m_t-\int_{\T^d}f\;\text{d}m_{\T^d}$ is given by the sum
	\begin{equation}
		\label{eq:expansion}
		\sum_{N\in \Z^{d}\setminus \{0\}}\hat{f}(N)\hat{\mu}(-tN)\;,
	\end{equation}
	for $\hat{\mu}(0)=1$ and $\hat{f}(0)=\int_{\T^{d}}f\;\text{d}m_{\T^d}$.
	
	In sum, we have translated the equidistribution problem for the measures $m_t$ into the Fourier-analytic question of the decay properties at infinity of the Fourier transform of the measure $\mu$.
Borrowing the terminology from~\cite{Strichartz}, we say that the Fourier transform of $\hat{\mu}$ decays along integral rays if 
	\begin{equation*}
		\lim\limits_{t\to\infty}\hat{\mu}(tN)=0
	\end{equation*}
	for any nonzero $N\in \Z^{d}$.

The previous computations lead us to the following characterization of equidistribution.  

	\begin{thm}
		\label{thm:qualitativeequidtorus}
		Let $d\geq 1$ be an integer, $(h_t)_{t\in \R_{>0}}$ the family of homotheties $h_t\colon \R^d\to \R^d$, $h_t(x)=tx$. Let $\mu$ be a Borel probability measure on $\R^{d}$,  $\pi\colon \R^d\to \T^d$ the canonical projection map, $m_t=(\pi\circ h_t)_*\mu$ for every $t>0$.
			 The measures $m_t$ equidistribute towards the Lebesgue measure $m_{\T^d}$ as $t\to\infty$ if and only if the Fourier transform of $\mu$ decays along integral rays. 
	\end{thm}
	\begin{proof}
		Sufficiency is clear from the expansion in~\eqref{eq:expansion}: if $\hat{\mu}$ decays along integral rays, then for every $f\in \mathrm{C}(\T^d)$ with $\hat{f}\in \ell^{1}(\Z^d)$ we have
		\begin{equation*}
		\lim\limits_{t\to\infty}\int_{\T^d}f\;\text{d}m_t-\int_{\T^d}f\;\text{d}m_{\T^d}=\lim\limits_{t\to\infty}\sum_{N\in \Z^{d}\setminus\{0\}}\hat{f}(N)\hat{\mu}(-tN)=\sum_{N\in \Z^{d}\setminus\{0\}}\hat{f}(N)\lim_{t\to\infty}\hat{\mu}(-tN)=0\;,
		\end{equation*} 
		the middle equality following from dominated convergence. Since the space $\mathscr{C}^{\infty}(\T^d)$ of smooth functions is dense in $\mathrm{C}(\T^d)$ and every smooth function has summable Fourier transform (cf.~\cite[Thm.~8.22]{Folland}), we deduce that the $m_t$ equidistribute towards $m_{\T^d}$.
		
		 Conversely, suppose the $m_t$ equidistribute towards $m_{\T^{d}}$ and let $N\in \Z^{d}\setminus\{0\}$. Consider the character $f(x+\Z^{d})=e^{2\pi i N\cdot (x+\Z^{d})}$ on $\T^d$; then
		\begin{equation*}
			0=\lim\limits_{t\to\infty}\int_{\T^{d}}f\;\text{d}m_t-\int_{\T^{d}}f\;\text{d}m_{\T^{d}}=\lim\limits_{t\to\infty}\int_{\T^{d}}e^{2\pi i N\cdot (x+\Z^{d})}\text{d}m_t(x+\Z^{d})=\lim\limits_{t\to\infty}\hat{\mu}(-tN)\;.
		\end{equation*}
	We conclude that $\hat{\mu}$ decays along integral rays.
	\end{proof}

	\begin{rmk}
		\label{rmk:generaltorusequid}
		We list here below a few comments concerning possible directions of generalization of Theorem~\ref{thm:qualitativeequidtorus}.
		\begin{itemize}
			\item[(a)] With obvious modifications, the proof of Theorem~\ref{thm:qualitativeequidtorus} yields the following more general statement: the measures $m_t$ equidistribute towards a given Borel probability measure $m$ as $t\to\infty$ if and only if $\lim_{t\to\infty}\hat{\mu}(-tN)\to \hat{m}(N)$ for every $N\in \Z^{d}\setminus \{0\}$. 
			\item[(b)] The characterization of equidistribution extends readily to the following version (cf.~\cite[Lem.~1]{Strichartz}). For any lattice $\Ga<\R^{d}$ and $t>0$, let $m_t^{(\Gamma)}$ denote the projection of $\mu_t=(h_{t})_*\mu$ onto $\R^{d}/\Ga$, and denote by $m_{\R^{d}/\Ga}$ the Haar probability measure on the torus $\R^{d}/\Ga$. Then the following are equivalent:
			\begin{itemize}
				\item [(i)] for any lattice $\Ga<\R^{d}$, the measures $m_{t}^{(\Ga)}$ equidistribute, as $t\to\infty$, towards $m_{\R^{d}/\Ga}$;
				\item [(ii)] the Fourier transform of $\mu$ decays along rays, that is, 
				\begin{equation*}
					\lim\limits_{t\to\infty}\hat{\mu}(tv)=0
				\end{equation*}
				for any $v\in \R^{d}\setminus\{0\}$.
			\end{itemize} 
			For this, simply observe that unitary characters for the compact group $\R^{d}/\Ga$ are of the form $x+\Ga\mapsto e^{2\pi i \eta\cdot (x+\Ga)}$, $\eta$ ranging over the dual lattice 
			\begin{equation*}
			\Ga^{*}=\{ v\in \R^d:v\cdot \gamma\in \Z \text{ for every }\gamma\in \Ga  \}\;;  
			\end{equation*}
			in complete analogy with what happens for the standard case $\Ga=\Z^d$, defining
			\begin{equation*}
				\hat{f}(\eta)=\int_{\R^{d}/\Ga}f(y+\Ga)e^{-2\pi i \eta\cdot (y+\Ga)}\text{d}m_{\R^{d}/\Ga}(y+\Ga)\;, \quad \eta \in \Ga^{*},
			\end{equation*}
			 we have that the infinite sum (Fourier series for $\Ga$)
			\begin{equation*}
				\sum_{\eta\in \Ga^{*}} \hat{f}(\eta) e^{2\pi i \eta\cdot (x+\Ga)}
			\end{equation*}
		of a smooth function $f$ on $\R^{d}/\Ga$ converges uniformly towards $f$, so that 
		\begin{equation*}
			\int_{\R^{d}/\Ga}f \;\text{d}m_{t}^{(\Ga)}=\sum_{\eta\in \Ga^{*}}\hat{f}(\eta)\int_{\R^{d}}e^{2\pi i \eta\cdot tx}\text{d}\mu_t(x)=\sum_{\eta\in \Ga^*}\hat{f}(\eta)\hat{\mu}(-t\eta)
		\end{equation*}
	for every $t>0$. 
			\item[(c)] Suppose the homotheties $h_t$ are centered at an arbitrary point $x_0\in \R^{d}$, so that explicitly they are given by $h_t(x)=x_0+t(x-x_0)$ for every $x\in \R^d$ and $t>0$. Then the expression for the discrepancy $\int_{\T^d}f\;\text{d}m_t-\int_{\T^d}f\;\text{d}m_{\T^d}$ in~\eqref{eq:expansion}, where now $m_t$ is the projection to $\T^d$ of the image of $\mu$ under the new homothety $h_t$, morphs into 
			\begin{equation*}
			\sum_{N\in \Z^{d}\setminus\{0\}}\hat{f}(N)e^{2\pi i N\cdot (1-t)x_0}\hat{\mu}(-tN)\;;
			\end{equation*}
		the additional factor $e^{2\pi i N\cdot (1-t)x_0}$ being of unit absolute value, the remaining arguments carry over unaffectedly. As a consequence, decay of the Fourier transform of $\mu$ along integral rays is again a necessary and sufficient condition for the $m_t$ to equidistribute towards $m_{\T^d}$. For the most general case of homotheties with linearly increasing ratio, we refer directly to the quantitative refinement in Remark~\ref{rmk:quantitativeremark}(c).
		\end{itemize}
	\end{rmk}

\subsection{Application: equidistribution of dilating submanifolds and self-affine fractals}
\label{subsec:qualitativeequid}
We would like to apply the equidistribution criterion enunciated in Theorem~\ref{thm:qualitativeequidtorus} to significative geometric examples of initial mass distributions $\mu$. We will discuss at considerable length the case where $\mu$ is the normalized surface measure on a smooth, compact, embedded submanifold of $\R^d$; for the sake of illustration, we shall also say a few words about the complementary case of measures supported on notable classes of fractal sets.   

\pagebreak
The notion of Rajchman measure encompasses all the examples we consider in the two cases just mentioned. We recall that a Borel probability measure $\mu$ on $\R^d$ is called a Rajchman measure\footnote{For our purposes, we are only interested in probability measures on $\R^d$, but the Rajchman property makes obvious sense for complex Radon measures on an arbitrary locally compact abelian topological group.} if its Fourier transform vanishes at infinity, that is, if $\lim_{\xi\to\infty}|\hat{\mu}(\xi)|=0$. The decay rate of Fourier transforms of measures has long been the subject of intensive research, starting essentially with the classical Riemann-Lebesgue lemma\footnote{Applying the Riemann-Lebesgue lemma in $\R^d$, we get at once that dilations of measures that are absolutely continuous with respect to the Lebesgue measure $m_{\R^d}$ equidistribute on the torus. From a geometric perspective, however, it is far more meaningful to consider the problem of equidistribution for dilations of  sets which are negligible from the point of view of the Lebesgue measure, such as lower-dimensional submanifolds.} (see the original references~\cite{Riemann,Lebesgue}) on the decay of Fourier coefficients of a periodic integrable function on the real line.  The study of such asymptotic properties has been largely motivated by developments around Riemann's celebrated problem of uniqueness for trigonometric series. More recently, fresh impetus to the investigation of Rajchman measures, as well as of related quantitative versions of the notion, has been provided by the fractal-geometric problem of absolute continuity of self-affine measures. We refer the reader to Lyons' survey~\cite{Lyons} for an historically-oriented introduction to the subject, and to the recent article of Li and Sahlsten~\cite{Li-Sahlsten}, presenting seminal progress on the question of Fourier decay for self-affine measures,  for a rich bibliography on more recent developments. 

Let us now resume with our equidistribution problem, and consider the surface measure $\mu$ on a smooth, compact, embedded submanifold $S\subset \R^d$, normalized to be a probability measure.  The following example demonstrates how curvature might pose obstructions to uniform distribution of expanding translates, and \emph{a fortiori} (in light of Theorem~\ref{thm:qualitativeequidtorus}) to Fourier decay\footnote{More about the interplay between curvature and harmonic analysis in certain problems can be found in the survey~\cite{Stein-second} by Stein and Wainger.}.

			\begin{ex}
			\label{ex:counterexample}
			For the sake of illustration, suppose we start with the uniform probability measure $\lambda$ on the open segment $(0,1)v=\{tv:t\in (0,1)  \}$, where $v=(v_1,\dots,v_d)\in \R^{d}$ is a non-zero vector. For every $t>0$, the projection to $\T^{d}$ of the dilation $h_t((0,1)v)=(0,t)v$ is the initial segment up to time $t$ of the forward orbit of the identity in $\T^d$ under the translational flow $\phi_t(x+\Z^{d})=tv+x+\Z^{d},\;x\in \R^{d}$, in direction $v$. It is well-known that such an orbit (or, for that matter, the orbit of any other initial point) equidistributes in $\T^{d}$ with respect to the Lebesgue measure if and only if the coordinates $v_1,\dots,v_d$ are linearly independent over $\Q$ (see, for instance, the survey~\cite{Kleinbock-Shah-Starkov} by Kleinbock, Shah and Starkov); a full quantitative understanding, depending on the Diophantine properties of the vector $v$, was provided by Green and Tao in~\cite{Green-Tao}. 
			
			As a matter of fact, it can be readily checked that, whenever $v_1,\dots,v_d$ are linearly dependent over $\Q$, the Fourier transform of $\lambda$ does not decay along at least one integral ray. More precisely, let $m=(m_1,\dots,m_d)\in \Z^{d}\setminus\{0\}$ be such that $m_1v_1+\cdots +m_dv_d=0$; then
		 \begin{equation*}
		 	\hat{\lambda}(tm)=\int_{\R^d}e^{-2\pi i tm\cdot x}\text{d}\lambda(x)=\int_{0}^{1}e^{-2\pi i tm \cdot sv}\text{d}s=\int_0^{1}e^{-2\pi ist (m\cdot v)}\text{d}s=1
		 \end{equation*}
		 for every $t>0$, whence there is no Fourier decay along the ray in direction $m$.
	
		More generally, let $\mu$ be a Borel probability measure on $\R^{d}$ whose support is contained in an affine hyperplane $V\subset \R^d$ defined over $\Q$, that is, which can be defined by a linear equation 
		\begin{equation*}
		m_1X_1+\cdots m_dX^d=m_{d+1} 
		\end{equation*}
		for some $m_1,\dots,m_d,m_{d+1}\in \Z$.
As before, it can be verified that $\hat{\mu}$  does not decay along the integral ray containing the non-zero vector $(m_1,\dots,m_d)\in \Z^d$.
An upshot of this discussion is therefore that compact submanifolds contained in a rational affine hyperplane do not equidistribute under dilation and projection to the standard torus.
\end{ex}		
	
\pagebreak
Under various assumptions pointing to non-vanishing features of the curvature, the Fourier transform of a surface measure $\mu$ as above does exhibit Fourier decay. A thorough discussion of the topic in its quantitative aspects is contained in Stein's book~\cite[Chap.~VIII, Sec.~3]{Stein}; here we shall content ourselves with stating a result for compact hypersurfaces. If $S\subset \R^d$ is a smooth compact hypersurface with the property that, for each point $x\in S$, at least one of the principal curvatures of $S$ does not vanish at $x$, then the normalized surface measure $\mu$ on $S$ is a Rajchman measure (see~\cite[Chap.~VIII, Thm.~2]{Stein}). We might therefore derive the following corollary from Theorem~\ref{thm:qualitativeequidtorus}. 
\begin{cor}
	\label{cor:qualitativesurfaces}
	Let $\mu$ be the normalized surface measure on a smooth, compact, embedded hypersurface $S\subset \R^d$. Suppose that, for every $x\in S$, at least one of the principal curvatures of $S$ does not vanish\footnote{It actually suffices that at least one of the principal curvatures doesn't vanish to infinite order at $x$, see~\cite[Chap.~VIII, Sec.~3.2]{Stein}.} at $x$. Then the measures $m_t=(\pi\circ h_t)_*\mu$ equidistribute towards the Lebesgue measure $m_{\T^d}$ as $t\to\infty$.
\end{cor}

The corollary provides a qualitative generalization of Randol's equidistribution result~\cite[Thm.~1]{Randol}, which is framed for compact hypersurfaces whose Gaussian curvature\footnote{Recall that the Gaussian curvature is the product of all the principal curvatures (see~\cite[Chap.~8]{Lee}).} is everywhere positive, and which are the boundary of compact convex subsets of $\R^d$ with non-empty interior. In Section~\ref{subsec:powerdecay}, we shall further give a quantitative refinement of Corollary~\ref{cor:qualitativesurfaces}, and thus a full extension of~\cite[Thm.~1]{Randol}.

\begin{rmk}
	\label{rmk:abscont}
	Randol's original article~\cite{Randol} includes a (quantitative) equidistribution result going beyond the case of hypersurfaces; specifically, it deals with dilating rectilinear $k$-simplices in $\R^d$, where $k$ is any integer between $1$ and $d$ (see~\cite[Thm.~2]{Randol}).  Availing of the full strength of ~\cite[Chap.~VIII, Thm.~2]{Stein}, Corollary~\ref{cor:qualitativesurfaces} admits likewise a more general formulation for compact embedded submanifolds of $S\subset \R^d$ of any positive dimension. The appropriate requirement to place on $S$ for the Fourier transform of its surface measure to decay at infinity is that $S$ is of finite type, that is, it has at most a finite order of contact with any affine hyperplane (see~\cite[Chap.~VIII, Sec.~3.2]{Stein} for a precise definition). We observe that, for a real-analytic manifold $S$, the condition amounts to $S$ not lying in any affine hyperplane, which is also necessary for Fourier decay to hold  (cf.~the discussion in Example~\ref{ex:counterexample}).
	
	Further, let us also mention that the aforementioned classical results on the Fourier decay of surface measures are not confined to the Riemannian volume measure inherited from the Euclidean structure on $\R^d$. In fact, they hold, and so does equidistribution of projected dilates as a result, for any measure which is absolutely continuous with smooth density with respect to the Riemannian volume measure on any submanifold subject to one of the previously detailed conditions (see~loc.~cit.). 
\end{rmk}

\begin{rmk}
	\label{rmk:Rajchman}
	Observe that the class of Rajchman measures on $\R^d$ is closed under absolute continuity: if $\mu_1$ is a Rajchman measure and $\mu_2$ is absolutely continuous with respect to $\mu_1$, then $\mu_2$ is also a Rajchman measure. This is a straightforward consequence\footnote{We hasten to point out that, in the case of the compact group $\T^d$, this inheritance of the Rajchman property can be proven rather easily and directly using density of trigonometric polynomials; for the details, we refer to~\cite[Sec.~2]{Lyons}. } of the following deep result established by Lyons~\cite{Lyons-abscont}: there exists a collection $\cE$ of Borel subsets of $\R^d$ with the property that a measure $\nu$ is Rajchman if and only if $\nu(E)=0$ for every $E\in \cE$. 
	
	From this, it follows immediately that, whenever dilates of a compact submanifold $S\subset \R^d$ equidistribute on $\T^d$, then so do those of any subset $S'\subset S$ having non-empty interior in $S$. For instance, magnified circle arcs equidistribute in $\T^2$. In Section~\ref{subsec:asymptotics}, we shall present the analogue of this fact in two-dimensional hyperbolic geometry.
\end{rmk}
	
\pagebreak		
Let us now turn briefly to the qualitative distribution features of expanding sets of fractal nature, encapsulated by the Rajchman property for natural measures supported on them. We consider the setting of self-affine measures, described as follows. Let $\cF=(F_j)_{j\in J}$ be an \emph{iterated function system}, that is, a finite collection of invertible affine maps $F_j\colon \R^d\to \R^d$, all of which are assumed to be contracting with respect to the Euclidean metric. We write $F_j(x)=A_j(x)+b_j$ for some $A_j\in \GL_d(\R)$ and $b_j\in \R^d$, $j\in J$. It is a classical fact, whose first general formulation appears in Hutchinson's article~\cite{Hutchinson}, that there exists a unique non-empty compact set $K\subset \R^d$ with the property that
\begin{equation*}
K=\bigcup_{j\in J}F_j(K)\quad ;
\end{equation*}
we shall refer to $K$ as the self-affine set generated by $\cF$. Furthermore, given a probability vector\footnote{This means that $p_j\in \R_{\geq 0}$ for every $j\in J$ and $\sum_{j\in J}p_j=1$. } $p=(p_j)_{j\in J}$, there exists a unique Borel probability measure $\mu$ on $\R^d$ such that 
\begin{equation*}
	\mu=\sum_{j\in J}p_j \;(F_j)_* \mu\quad ;
\end{equation*}
if $p_j>0$ for every $j\in J$, then $\mu$ is fully supported on $K$. We call $\mu$ the self-affine measure generated by the pair $(\cF,p)$. We refer again to~\cite{Hutchinson} for a proof of all these assertions. 

It has been recently shown by Li and Sahlsten in~\cite{Li-Sahlsten} that, under mild non-degeneracy assumptions on the iterated function system $\cF$, the self-similar measure $\mu$ generated by $(\cF,p)$ has the Rajchman property for every probability vector $p=(p_j)_{j\in J}$ with $0<p_j<1$ for every $j\in J$. Specifically, this holds whenever $K$ is not reduced to a singleton and the subgroup of $\GL_d(\R)$ generated by $\{A_j:j\in J  \}$ is proximal and totally irreducible (see~\cite[Thm.~1.1]{Li-Sahlsten}). 

\begin{rmk}
The homogeneous case, in which there exists $A\in \GL_d(\R)$ such that $A_j=A$ for every $j\in J$, has been considered by Solomyak in~\cite{Solomyak}, who established the following result. There exists a Lebesgue-null set $E\subset \R^d$ such that, for every vector $\theta=(\theta_1,\dots,\theta_d)\in \R^{d}\setminus E$ with $\inf_{1\leq i\leq d}|\theta_i|>1$, for every homogeneous iterated function system $\cF=(F_j)_{j\in J}$ with linear part the diagonal matrix $A=\text{diag}(\theta_1^{-1},\dots,\theta_d^{-1})$, and for every probability vector $p=(p_j)_{j\in J}$ with $0<p_j<1$ for every $j\in J$, the self-affine measure generated by the pair $(\cF,p)$ is a Rajchman measure provided that its self-affine support $K$ is not contained in any affine hyperplane. 

For ease of reading, we shall confine ourselves to the setting of Li and Sahlsten in the statement of the next corollary.
\end{rmk}

 We might therefore infer the following corollary of Theorem~\ref{thm:qualitativeequidtorus} and~\cite[Thm.~1.1]{Li-Sahlsten}.

\begin{cor}
	\label{cor:equidfractals}
	Let $d> 1$, $\cF=(F_j)_{j\in J}$ a finite collection of invertible affine contractions on $\R^d$, $p=(p_j)_{j\in J}$ a probability vector with $0<p_j<1$ for all $j\in J$, $\mu$ the self-affine measure generated by $(\cF,p)$.  Suppose that the self-affine set generated by $\cF$ is not reduced to a singleton, and that the subgroup of $\GL_d(\R)$ generated by the linear parts of $\cF$  is proximal and totally irreducible. Then the measures $m_t=(\pi\circ h_t)_*\mu$ equidistribute towards the Lebesgue measure $m_{\T^d}$ as $t\to\infty$.
\end{cor}

As for the case of surface measures, we shall improve upon Corollary~\ref{cor:equidfractals} quantitatively in Section~\ref{subsec:powerdecay}, relying upon effective power decay estimates, established by Li and Sahlsten, on the Fourier transform of a broad class of self-affine measures.

We conclude our investigation of equidistribution properties of dilating sets in Euclidean spaces by placing emphasis on the fact that, as already hinted at in Example~\ref{ex:counterexample}, the Rajchman property for the initial measure $\mu$ is by no means necessary for uniform distribution to hold: decay of the Fourier transform along integral rays, which according to Theorem~\ref{thm:qualitativeequidtorus} is the actual characterizing feature of $\mu$ for its dilates to equidistribute, is a substantially weaker property. For instance, it is satisfied whenever the orthogonal projection of $\mu$ onto any rational line in $\R^d$ is absolutely continuous with respect to the standard Lebesgue measure on that line. 
This is a straightforward consequence of the classical Riemann-Lebesgue lemma~\cite{Riemann,Lebesgue}: if, for some $v\in \R^{d}$ with rationally dependent coordinates, the orthogonal projection $(\pi_{\ell})_*\mu$ of $\mu$ onto the rational line $\ell=\{av:a\in \R \}$ is absolutely continuous with respect to the Lebesgue measure $m_{\ell}$ on $\ell$, then, assuming for convenience that $v$ has unit Euclidean norm,
	\begin{equation*}
		\hat{\mu}(tv)=\int_{\R^{d}}e^{-2\pi i tv\cdot x}\text{d}\mu(x)=\int_{\R^d }e^{-2\pi i t \pi_{\ell}(x)}\text{d}\mu(x)=\int_{\ell}e^{-2\pi i t y}\text{d}(\pi_{\ell})_*(y)\overset{t\to\infty}{\longrightarrow}0\;,
	\end{equation*}
so that $\hat{\mu}$ decays along the ray generated by $v$.

\begin{rmk}
	\begin{itemize}
	\item[(a)] For the implications of this fact for the problem of equidistribution of magnified curves, as well as for related open questions, we refer to~\cite{Strichartz}. 
	\item[(b)] As in the case of Rajchman measures (see~Remark~\ref{rmk:Rajchman}), the class of measures with absolutely continuous projections along rational lines is ostensibly closed under absolute continuity.  
	\end{itemize}
\end{rmk}

	\section{Power Fourier decay and rates of equidistribution}
\label{sec:quantitativeeuclidean}

\subsection{Fourier dimension and effective equidistribution} The argument presented in Section~\ref{sec:qualitativetorus}, leading to qualitative equidistribution under the assumption of Fourier decay along integral rays, is amenable to a quantitative refinement, provided that information on the rate of decay of the Fourier transform of the original measure $\mu$ is available.
	
 In what follows, we denote by $|x|$ the Euclidean norm of a vector $x\in \R^d$.  	If $a\colon \Z^d\to \C$ is a function in $\ell^{q}(\Z^{d})$ for some  $q\in \R_{\geq 1}$, we let
\begin{equation*}
	\norm{a}_{\ell^{q}(\Z^d)}=\biggl(\sum_{N\in \Z^{d}}|a(N)|^{q}\biggr)^{1/q}\;.
\end{equation*}

	The effective version of Theorem~\ref{thm:qualitativeequidtorus} reads as follows.
	\begin{thm}
		\label{thm:quantitativetorus}
		Let $\mu, (h_t)_{t>0}$ and $(m_t)_{t>0}$ be as in Theorem~\ref{thm:qualitativeequidtorus}. 
		Assume that there exists $s\in \R_{\geq 0}$ such that the Fourier transform of $\mu$ satisfies $|\hat{\mu}(\xi)|=O(|\xi|^{-s/2})$. Then there exists \linebreak$C\in \R_{>0}$, depending only on $\mu$, such that, for any continuous function $f\colon \T^d\to \C$ with summable Fourier transform,
		\begin{equation}
			\label{eq:quantitative}
			\biggl|\int_{\T^{d}}f\;\emph{d}m_t-\int_{\T^{d}}f\;\emph{d}m_{\T^{d}}\biggr|\leq C\;\bigl\lVert\hat{f}\bigr\rVert_{\ell^{1}(\Z^d)}\;t^{-s/2}
		\end{equation}
		for every $t>0$.
	\end{thm}

The choice of $s/2$ as a parametrization of the polynomial rate of decay of the Fourier transform of $\mu$ is related to the notion of Fourier dimension of a measure, to which we shall shortly turn in the context of surface measures and self-affine measures.

	\begin{proof}
		The assumption on the decay rate of $\hat{\mu}$, combined with the elementary fact that the Fourier transform of a measure is a bounded function, implies that there exists $C=C(\mu)\in \R_{>0}$ such that $|\hat{\mu}(\xi)|\leq C|\xi|^{-s/2}$ for any $\xi\in \R^{d}$.
		If $f\colon \T^{d}\to \C$ is a continuous function with $\hat{f}\in \ell^{1}(\Z^d)$, then we deduce from the expansion in~\eqref{eq:integralexpansion} that 
		\begin{equation*}
			\begin{split}
				\biggl|\int_{\T^{d}}f\;\text{d}m_t-\int_{\T^{d}}f\;\text{d}m_{\T^{d}} \biggr|&=\biggl|\sum_{N\in \Z^{d}\setminus\{0\}}\hat{f}(N)\hat{\mu}(-tN)\biggr|\leq C\;\bigl\lVert\hat{f}\bigr\rVert_{\ell^{1}(\Z^d)}\sup_{N\in \Z^{d}\setminus\{0\}}|tN|^{-s/2}\\
				&= C\;\lnorm{\hat{f}}_{\ell^{1}(\Z^d)}\;t^{-s/2}\sup_{N\in \Z^{d}\setminus \{0\}}|N|^{-s/2}\\
				&= C\; \lnorm{\hat{f}}_{\ell^{1}(\Z^d)}\;t^{-s/2}
			\end{split}
		\end{equation*}
	for every $t>0$, 
		as desired.
	\end{proof}

	\begin{rmk}
		\label{rmk:quantitativeremark}
		\begin{itemize}
		\item[(a)]Let $\mu$ fulfill the same Fourier-decay assumption as in Theorem~\ref{thm:quantitativetorus}, set $p_0=\frac{2d}{2d-s}$ and suppose, less stringently, that $f\colon \T^{d}\to \C$ is  continuous function with $\hat{f}\in \ell^{p}(\Z^d)$ for some $1\leq p<p_0$.  If $q$ is the H\"{older}-conjugate exponent of $p$, then H\"{o}lder's inequality gives
		\begin{equation*}
			\biggl|\int_{\T^{d}}f\;\text{d}m_t-\int_{\T^{d}}f\;\text{d}m_{\T^{d}} \biggr|\leq C\; \lnorm{\hat{f}}_{\ell^{p}(\Z^d)}\;t^{-s/2}\biggl(\sum_{N\in \Z^{d}\setminus\{0\}}|N|^{-sq/2}\biggr)^{1/q}\;,
		\end{equation*}
	where $C$ is the constant appearing in the statement of Theorem~\ref{thm:quantitativetorus}.
		The condition $p<p_0$ is equivalent to $sq>2d$, so that the infinite sum on the right-hand side of the last inequality converges. We thus have the same estimate as in~\eqref{eq:quantitative} with the $\ell^{1}$-norm replaced by the $\ell^{p}$-norm, at the cost of increasing the value of the constant $C$.
	\item[(b)]	The condition $\hat{f}\in \ell^{1}(\Z^d)$ is satisfied whenever $f\in \mathscr{C}^{d+1}(\T^d)$, the space of $(d+1)$-times continuously differentiable functions on $\T^d$  (cf.~\cite[Thm.~8.22]{Folland}); moreover, in such a case the $\ell^{1}$-norm of $\hat{f}$ is majorized, up to a multiplicative constant depending only on $d$, by the $\mathscr{C}^{d+1}$-norm of $f$, the latter being defined as the sum of the uniform norms of all mixed partial derivatives of $f$ up to order $d+1$.
	\item[(c)] The statement of Theorem~\ref{thm:quantitativetorus} is unaffected when replacing the standard linear dilations $h_t(x)=tx$ by a general family of homotheties $h_t(x)=tA_t(x)+b_t$, where $A_t\in \Or_d(\R)$ and $b_t\in \R^{d}$.  Indeed, the expansion in~\eqref{eq:integralexpansion} changes into 
	\begin{equation*}
		\int_{\T^d}f\;\text{d}m_{t}=\sum_{N\in \Z^{d}}\hat{f}(N)e^{2\pi i N\cdot b_t}\hat{\mu}(-t \;A_t^{-1}(N))\;,
	\end{equation*}
as an immediate computation allows to verify. Consequently,
\begin{equation*}
	\begin{split}
	\biggl|\int_{\T^d}f\;\text{d}m_t-\int_{\T^d}f\;\text{d}m_{\T^d}\biggr|
	&=\biggl|\sum_{N\in \Z^{d}\setminus\{0\}}\hat{f}(N)e^{2\pi i N\cdot b_t}\hat{\mu}(-t\;A_t^{-1}(N))\biggr|\\
	&\leq C\; \lnorm{\hat{f}}_{\ell^{1}(\Z^d)}\;t^{-s/2}\supl_{N\in \Z^d\setminus \{0\}} |A_t^{-1}(N)|^{-s/2}\\
	&= C\;\lnorm{\hat{f}}_{\ell^{1}(\Z^d)}\;t^{-s/2}\supl_{N\in \Z^d\setminus \{0\}} |N|^{-s/2}\\
	&=C\;\lnorm{\hat{f}}_{\ell^1(\Z^d)}\;t^{-s/2}\;,
	\end{split}
\end{equation*}
the third equality holding simply because $A_t$ is norm-preserving. 
		\end{itemize}
	\end{rmk}

\subsection{Power decay for surface and self-affine measures}

\label{subsec:powerdecay}
	The formulation of Theorem~\ref{thm:quantitativetorus} behooves us to examine more closely the geometric significance of the condition\linebreak $|\hat{\mu}(\xi)|=O(|\xi|^{-s/2})$. A well-established theory in fractal geometry, revolving around the notions of Fourier dimension of sets and measures in Euclidean spaces, indicates that a polynomial Fourier decay rate of exponent $-s/2$ ought to correspond to $\mu$ having dimension $s$. The connection is perhaps more transparent in the case of sets, thanks to the well-known equivalence of Hausdorff and capacitary dimension for Borel sets (see~\cite[Thm.~8.9]{Mattila}): if $\dim_H(A)$ denotes the Hausdorff dimension of a Borel set $A\subset \R^d$, defined as 
	\begin{equation*}
		\dim_H(A)=\;\inf\biggl\{t\in \R_{\geq 0}: \forall\;\eps>0\;\; \exists \;E_i\subset \R^d,\;i\in \N\text{ s.t. }A\subset \bigcup_{i\in \N}E_i \text{ and }\sum_{i\in \N}\text{diam}(E_i)^{t}<\eps  \biggr\}
	\end{equation*}
where $\text{diam}(E_i)=\sup\{|x-y|:x,y\in E_i  \}$, then it is a consequence of Frostman's lemma (see, e.g.,~\cite[Thm.~8.8]{Mattila}) that
\begin{equation}
	\label{eq:Hausdorffcapacitary}
	\dim_H(A)=\sup\;\biggl\{s\in \R_{\geq 0}: \exists\;\nu \in \cM(A) \text{ s.t. }  \int_{\R^d}\int_{\R^d}|x-y|^{-s}\;\text{d}\nu(x)\text{d}\nu(y)<\infty  \biggr\}\;,
\end{equation}
where $\cM(A)$ is the set of Borel probability measures whose support is compact and contained in $A$, and we conventionally agree that $|x-y|^{-s}=+\infty$ whenever $x=y$ in the last integrand. As a consequence of the formula (see~\cite[Thm.~3.10]{Mattila-second})
\begin{equation*}
	\int_{\R^d}\int_{\R^d}|x-y|^{-s}\;\text{d}\nu(x)\text{d}\nu(y)=c(d,s)\int_{\R^{d}}|\hat{\nu}(\xi)|^{2}|\xi|^{s-d}\;\text{d}\xi\;, \quad 0<s<d,
\end{equation*}
$c(d,s)$ being a positive real number only depending on $d$ and $s$, it follows from the equality in~\eqref{eq:Hausdorffcapacitary} that
the Hausdorff dimension of a Borel set $A\subset \R^d$ is at least as large as
\begin{equation*}
	\sup\;\{0\leq s\leq d:\exists\; \nu \in \cM(A) \text{ s.t. }|\hat{\nu}(\xi)|=O(|\xi|^{-s/2})   \}\;,
\end{equation*}
with the latter quantity being customarily defined as the Fourier dimension of $A$, hereafter denoted $\dim_F(A)$. Borel sets for which the Hausdorff and the Fourier dimension coincide are called Salem sets. Producing deterministic examples of those is remarkably challenging; the reader is referred to~\cite[Sec.~3.6]{Mattila-second} for examples and references.

As regards compact submanifolds of $\R^d$,  subtleties emerge once more because of curvature issues. The theory of oscillatory integrals in harmonic analysis offers satisfactory answers mostly in the case of hypersurfaces.  If $S$ is a smooth, compact, embedded hypersurface in $\R^{d}$ whose Gaussian curvature vanishes nowhere, then $S$ is a Salem set. Specifically, the supremum in the definition of the Fourier dimension is attained by the volume measure $\mu$ on $S$, which satisfies $|\hat{\mu}(\xi)|=O(|\xi|^{-(d-1)/2})$ (see~\cite[Chap.~VIII, Thm.~1]{Stein}). As a result, Theorem~\ref{thm:quantitativetorus} readily yields the following effective version of Corollary~\ref{cor:qualitativesurfaces}, which is also a generalization of~\cite[Thm.~1]{Randol}.

\begin{cor}
	\label{cor:dilatingsurfaces}
	Let $\mu$ be the normalized surface measure on a smooth, compact, embedded hypersurface $S\subset \R^d$ with nowhere vanishing Gaussian curvature, $m_t=(\pi\circ h_t)_{*}\mu$ for every $t>0$. Then there exists $C>0$, depending only on $S$, such that, for any continuous function $f\colon \T^d\to \C$ with summable Fourier transform,
	\begin{equation}
		\label{eq:difference}
		\biggl|\int_{\T^{d}}f\;\emph{d}m_t-\int_{\T^{d}}f\;\emph{d}m_{\T^{d}}\biggr|\leq C\;\lnorm{\hat{f}}_{\ell^{1}(\Z^d)}\;t^{-(d-1)/2}
	\end{equation}
	for every $t>0$.
\end{cor}
	
We have thus shown that the equidistribution rate of expanding hypersurfaces is polynomial in the expansion rate $t$, with an exponent improving linearly with the dimension $d$.

\begin{rmk}
	Under the assumptions of Corollary~\ref{cor:dilatingsurfaces}, it is possible to refine the effective estimate $|\hat{\mu}(\xi)|=O(|\xi|^{-(d-1)/2})$ by means of an asymptotic expansion at infinity for the Fourier transform, which in turn translates into an asymptotic expansion for the discrepancy on the left-hand side of~\eqref{eq:difference}. For a discussion of such matters, see~\cite[Chap.~VIII, Sec.~5.7]{Stein}. 
	
	In Section~\ref{subsec:asymptotics}, we shall similarly present a precise asymptotic expansion, in the framework of two-dimensional hyperbolic geometry, for averages of sufficiently regular observables along expanding pieces of homogeneous curves.
\end{rmk}	

	Suppose now, more generally, that $\mu$ is the surface measure on a smooth, compact, embedded hypersurface $S\subset \R^{d}$ such that at least $k$ ($1\leq k\leq d-1$) of its principal curvatures vanish nowhere. Generalizing what has just been stated in the case $k=d-1$,  Littman~\cite{Littman} proved that  $|\hat{\mu}(\xi)|=O(|\xi|^{-k/2})$. The ensuing effective equidistribution statement for the projected dilates $m_t$ with polynomial decay rate $k/2$ follows as before.

The case of a compact submanifold of arbitrary intermediate dimension $1\leq m\leq d-1$\linebreak is notoriously more delicate, as the Fourier decay features of a measure crucially depend on the way its support is embedded inside $\R^d$; in other words, the Fourier dimension is not invariant under isometries between Euclidean spaces. Power decay results exist (for instance, see\linebreak\cite[Chap.~VIII, Thm.~2]{Stein}), but the exponent does not match the dimension of the manifold in the same way it does for $m=d-1$ and nonvanishing Gaussian curvature.
	
\pagebreak	
\begin{rmk}
	As already observed for the non-effective equidistribution statements in Section~\ref{subsec:qualitativeequid} (see, in particular, Remark~\ref{rmk:abscont}), Corollary~\ref{cor:dilatingsurfaces} applies verbatim, up to modifying the constant $C$ appropriately, to any absolutely continuous measure with smooth density with respect to the surface measure; for this, we invoke the full generality of the quantitative Fourier decay in~\cite[Chap.~VIII, Thm.~1]{Stein}.
\end{rmk}

To mirror our foregoing treatment of the qualitative aspects of equidistribution, let us terminate this section with a brief mention of the much less understood case of self-affine sets, of which we only scratch the surface. The following power-decay result is~\cite[Thm.~1.2]{Li-Sahlsten}: suppose $K$ is a self-affine set, not reduced to a singleton, determined by a family $\cF=(F_j)_{j\in J}$ of affine contractions with the property that the Zariski closure of the subgroup of $\GL_d(\R)$ generated by the linear parts of $\cF$  is a connected $\R$-splitting reductive group acting irreducibly on $\R^d$. Then, for any self-affine measure determined by $\cF$ and fully supported on $K$, there exists $\alpha>0$ such that $|\hat{\mu}(\xi)|=O(|\xi|^{-\alpha})$. Therefore, expanding self-affine sets satisfying the previous property equidistribute on $\T^d$ with a polynomial rate. 

\begin{rmk}
	For the case of homogeneous self-affine sets, see Solomyak's results on power Fourier decay in~\cite{Solomyak}. 
\end{rmk}

	\section{Effective equidistribution of expanding spheres on hyperbolic manifolds}
	
	\label{sec:hyperbolic}
	
	Having thoroughly discussed the equidistribution problem of present interest in the Euclidean case, we now switch to the setting of hyperbolic manifolds. The foundations of their rich and multifaceted theory are laid, for instance, in~\cite{Ratcliffe}.
	
	Up to Riemannian isomorphisms, a compact, connected, orientable, hyperbolic $d$-dimensional manifold ($d\geq 2$) is a quotient $M=\Ga\bsl \Hyp^{d}$ of the upper half space
	\begin{equation*}
		\Hyp^{d}=\{ (x_1,\dots,x_{d-1},y)\in \R^{d}:y>0 \}\;,
	\end{equation*}
	endowed with its canonical Riemannian metric of sectional curvature constantly equal to $-1$, by a torsion-free cocompact lattice $\Ga$ in the group of orientation-preserving isometries of $\Hyp^{d}$, isomorphic to the connected component $\SO_{d,1}^{+}(\R)$ of $\SO_{d,1}(\R)$. The hyperbolic $d$-space $\Hyp^{d}$, a Riemannian universal covering space of $M$, plays here the role taken on by $\R^d$ in the foregoing sections. We choose arbitrarily a base point $x\in \Hyp^{d}$ and, for every $t\in \R_{>0}$, we let $h_t\colon \Hyp^{d}\to \Hyp^{d}$ be the unique homothety of ratio $t$ which fixes the point $x$ and whose differential acts as a central homothety on the tangent space $T_x\Hyp^d$.
As in the Euclidean case, the map $h_t$ admits the following explicit description. For every $y\in \Hyp^{d}$, let $\gamma\colon \R \to \Hyp^{d}$ be the unique unit-speed geodesic passing through $y$ with starting point $\gamma(0)=x$. Then $h_t(q)=\gamma(t d(x,y))$, where $d$ is the hyperbolic Riemannian distance function on $\Hyp^d$.

	Given a Borel probability measure $\mu$ on $\Hyp^d$, we enquire once more about the asymptotic behaviour of the measures $m_t=(\pi\circ h_t)_{*}\mu$, where $\pi\colon \Hyp^{d}\to M$ is a covering map. A case in point is given by taking $\mu$ to be the normalized volume measure on the unit sphere 
	\begin{equation*}
		S(x,1)=\{ y\in \Hyp^{d}:d(x,y)=1 \}\;,
	\end{equation*}
	equipped with the induced Riemannian structure.
The measure $m_t$ is then the projection to $M$ of the natural volume measure on the sphere $S(x,t)$ of radius $t$ centered at $x$; as such, it is supported on the geodesic sphere $\mathcal{S}(\pi(x),t)$ centered at $\pi(x)$, defined as the set\footnote{In other words, $\mathcal{S}(\pi(x),t)$ is the image of the radius-$t$ circle centered at the origin in the tangent space $T_{\pi(x)}M$ under the Riemannian exponential map $T_{\pi(x)}M\to M$.} 
	\begin{equation}
		\label{eq:geodesiccircle}
		\mathcal{S}(\pi(x),t)=\{ \gamma_{\pi(x),v}(t):v\in T^{1}_{\pi(x)}M  \}\;,
	\end{equation}
where $T^{1}_{\pi(x)}M$ denotes the unit tangent space to $M$ at $\pi(x)$ and $\gamma_{\pi(x),v}$ is the unique Riemannian geodesic on $M$ with $\gamma_{\pi(x),v}(0)=\pi(x)$ and $\gamma_{\pi(x),v}'(0)=v$.

	\begin{rmk}
	The geodesic sphere of radius $t$ centered at a point $\pi(x)\in M$ does not coincide, in general, with the hyperbolic sphere of same radius and center, defined as the set of points in $M$ at hyperbolic distance $t$ from $\pi(x)$. This is because geodesics are only locally distance-minimizing curves\footnote{ The disparity of the two notions is manifest when $t$ is larger than the diameter of $M$: every hyperbolic sphere of radius $t$ is the empty set, while this is never the case for geodesic spheres, by geodesic completeness of $M$.}. 
\end{rmk}

	 This example has been investigated by Randol in~\cite{Randol}, who offers a complete argument for the following effective equidistribution statement in $d=3$, quantifying the fact that geodesic spheres of increasing radius tend to distribute uniformly inside $M$. Let $\Delta_M$ be the Laplace-Beltrami operator determined by the hyperbolic structure on $M$ (we refer to the book~\cite{Berger} of Berger, Gauduchon and Mazet for a comprehensive introduction to spectral geometry). Letting $\vol_M$ be the hyperbolic volume measure on $M$, normalized to be a probability measure, the spectrum of $\Delta_M$ acting as an unbounded operator on the space $L^{2}(M,\vol_M)$ is, by compactness of $M$, a discrete subset of $\R_{\geq 0}$ only consisting of eigenvalues, which we enumerate as
	\begin{equation*}
		0=\lambda_0<\lambda_1\leq \lambda_2\leq \cdots \leq \lambda_n\leq \cdots \;,
	\end{equation*}
	taking the possible (finite) multiplicities into account. As is custom, we write $\lambda_n=1+r_n^{2}$ with $r_n\in \R_{\geq 0}\cup i \R_{>0}$. Then Randol's result asserts that, for every smooth function $f\colon M\to \C$, there exists $C>0$, depending on $M$ and $f$, such that 
	\begin{equation*}
	\biggl|	\int_{M}f\; \text{d}m_t-\int_{M}f\;\text{d}\vol_M\biggr|\leq Ce^{-\alpha t}
	\end{equation*}
	for every $t>0$, where $\alpha=1$ if $\lambda_1\geq 1$ and $\alpha=1-|r_1|$ if $\lambda_1<1$.  	

	Observe both the qualitative analogy and the quantitative dissimilarity with the Euclidean setting. In the latter, expanding spheres equidistribute with a polynomial rate (cf.~Corollary~\ref{cor:dilatingsurfaces}), whereas here equidistribution occurs at an exponential rate, with the exponent $\alpha$ depending explicitly on the spectral gap $\lambda_1$ of the the hyperbolic manifold $M$. We shall shortly dwell a little more on the relevance of the spectral gap in the quantification of equidistribution phenomena of this sort (see Section~\ref{subsec:asymptotics}).

	Randol's argument in~\cite{Randol}, which admits a straightforward but computationally more tedious generalization to any dimension $d\geq 2$, relies fundamentally on the spherical symmetry of the measures $\mu_t=(h_t)_*\mu$ around the center of the dilations, and is based upon techniques related to the Selberg trace formula; we refer to the extensive discussion in~\cite[Chap.~XI, Sec.~3]{Chavel} for the details, as well as for numerous other applications of the same techniques. 
	
	\begin{rmk}
	Dilating geodesic spheres in hyperbolic three-manifolds have also been studied by Peter Sarnak via the wave equation. Though 
	this has never appeared in print, as far as the author can say, we refer to~\cite{Klainerman-Sarnak} for closely related developments. 
	\end{rmk}
	
	\section{Expanding translates of homogeneous curves on hyperbolic surfaces}
	\label{sec:unittangentbundles}
	
	In the present manuscript, we shall also pursue a spectral approach, albeit of a different nature, in order to achieve a precise asymptotic expansion for averages of smooth functions along canonical lifts of expanding geodesic circles to the unit tangent bundle $T^{1}M$ of a compact connected hyperbolic surface $M$. Upgrading the problem to the bundle, which can be identified with a homogeneous space of the isometry group $\SO_{2,1}^+(\R)$, allows to bring in tools, admittedly of a rather elementary nature, from the theory of unitary representations of semisimple Lie groups. These, in combination with an ingenious strategy first devised by Ratner~\cite{Ratner} in her study of quantitative mixing features of the geodesic and horocycle flows on finite-volume Riemann surfaces, enable a fine asymptotic analysis of the equidistribution properties of expanding circle arcs, thereby dispensing with the Randol's spherical symmetry assumption. Furthermore, the same argument yields, more generally, analogous asymptotic expansions for averages along dilating segments of any homogeneous curve. The latter represents the only novelty with respect to the results already obtained in our earlier work~\cite{Corso-Ravotti} with Ravotti on large hyperbolic circles.  
	
	\begin{rmk}
		Ratner's insightful approach was subsequently developed in several directions; we mention here the works of Burger~\cite{Bur}, Strömbergsson~\cite{Str}, Edwards~\cite{Edwards} and Ravotti~\cite{Rav}, all pertaining to quantitative equidistribution problems which are intimately tied to those considered here. 
	\end{rmk}
	
	We expect the overarching strategy to be sufficiently flexible to furnish similar results in every dimension $d\geq 2$, and plan to investigate the matter in future work. In addition, we intend to examine to which extend the method produces relevant information on the quantitative equidistribution properties of expanding sets beyond the homogeneous case, with a particular emphasis on submanifolds and self-similar sets. Beside their intrinsic geometric interest, results of this sort potentially bear implications for questions in the metric theory of Diophantine approximation; for instance, in the case of self-similar sets, the reader may consult the recent work of Khalil and Luethi~\cite{Khalil-Luethi} for a prominent example of such implications.
	
	 Further directions of research the present work breeds are outlined in Section~\ref{sec:further}.

	\subsection{The setup: compact hyperbolic surfaces, their unit tangent bundles and the action of the geodesic flow}
	We now set about illustrating the refinement of our equidistribution problem for expanding circles in the setup of unit tangent bundles of compact hyperbolic surfaces. For a more extensive presentation of the required background, including the relevant references, we refer to~\cite[Sec.~2]{Corso-Ravotti}.  
	
	We simply write $\Hyp$ in place of $\Hyp^{2}$ for the Poincaré upper-half plane model of the hyperbolic plane. A compact, connected, orientable hyperbolic surface $M$ is identified with a quotient $\Ga\bsl \Hyp$, where $\Ga$ is a uniform lattice in $\SL_2(\R)$\footnote{More precisely, $\Ga$ arises as the preimage of a torsion-free uniform lattice in $\PSL_2(\R)=\SL_2(\R)/\{\pm \text{I}_2 \}$, which is isomorphic to the full group $\SO_{2,1}^{+}(\R)$ of orientation-preserving isometries of $\Hyp$.} and the latter Lie group acts by isometries of the hyperbolic metric on $\Hyp$ via the linear fractional transformations
	\begin{equation}
		\label{eq:isometricaction}
		\begin{pmatrix}
			a&b\\c&d
		\end{pmatrix}
	\cdot z=\frac{az+b}{cz+d}\;, \quad z=x+iy\in \Hyp, \;a,b,c,d\in \R, \;ad-bc=1.
	\end{equation}

	The identification of $M$ with $\Ga\bsl \Hyp$ extends to an identification of the unit tangent bundle $T^1M$ with the compact homogeneous space $Y=\Ga\bsl \SL_2(\R)$.
	We let $\mathbf{p}\colon T^1M\to M$ be the canonical projection map, and denote by $m_Y$ the unique $G$-invariant Borel probability measure on $M$; alternatively, this is defined as the Liouville measure on $T^{1}M$ projecting to the normalized hyperbolic area measure\footnote{As suggested by the notation we have been keeping with throughout, $m_Y$ can further be seen as a Riemannian volume measure, namely as the one determined by the Sasaki metric (see~\cite{Sasaki}) on the unit tangent bundle $T^1M$ arising from the hyperbolic metric on $M$.} $\vol_M$ on $M$. The geodesic flow $(\phi_t)_{t\in \R}$ on $T^1M$ transports each unit tangent vector at unit speed along the unique geodesic it defines on $M$. More precisely, it is defined as 
	\begin{equation*}
		\phi_t\colon T^1M\to T^1M\;,\quad (z,v)\mapsto (\gamma_{z,v}(t),\gamma'_{z,v}(t))\;, \quad z\in M,\;v\in T^1_zM, \;t\in \R,
	\end{equation*}
where $\gamma_{z,v}\colon \R \to M$ is the unique geodesic on $M$ with $\gamma_{z,v}(0)=z$ and $\gamma_{z,v}'(0)=v$. Crucially for our approach, the geodesic flow admits, in the setting of constant negative curvature, a purely algebraic description, in this case as a homogeneous flow on $Y$: specifically, under the identification of $T^1M$ with $Y$ it is given by 
	\begin{equation*}
		\phi_t(\Ga g)=\Ga g\begin{pmatrix}e^{t/2}&0\\0&e^{-t/2}\end{pmatrix}
	\end{equation*} 
	for any $g\in \SL_2(\R)$ and $t\in\R$. More generally for our forthcoming purposes, let us define, for every element $W$ in the Lie algebra $\sl_2(\R)$ of $\SL_2(\R)$, the homogeneous flow $(\phi_s^{W})_{s\in \R}$ given by
	\begin{equation}
		\label{eq:homogeneousflow}
		\phi_s^{W}(\Ga g)=\Ga g \exp(sW)\;, \quad g\in \SL_2(\R), \;s\in \R,
	\end{equation}
where $\exp\colon \sl_2(\R)\to \SL_2(\R)$ is the exponential map. Observe that the geodesic flow itself belongs to this class of flows, namely $\phi_t=\phi_t^{X}$ for $
	\label{eq:X}
	X=\begin{pmatrix}
		1/2&0\\
		0&-1/2
	\end{pmatrix} $.

 In the sequel, we refer to integral curves\footnote{By a mild abuse, the terminology of curve stands here both for a map $\gamma\colon I\to N$ from an interval $I\subset \R$ to a manifold $N$ and for its image $\{\gamma(s):s\in I\}\subset N$.  } of flows of the type $(\phi_s^{W})_{s\in \R}$, for $W\in \sl_2(\R)$, as well as to any restriction of those to subintervals of $\R$, as homogeneous curves on $Y$.
	
	Now define
	\begin{equation}
		\label{eq:theta}
		\Theta=\begin{pmatrix}
			0&1\\-1&0 
		\end{pmatrix}
	\in \sl_2(\R)
	\end{equation}
and	observe that, for every point $q\in Y$, the curve $\{\phi^{\Theta}_s(q):0\leq s\leq 2\pi  \}$ parametrizes the unit tangent space $T^1_{z}M$ at $z=\mathbf{p}(q)$. This follows from the analogous observation for the universal cover $\Hyp$, which in turn is a rather direct consequence of the fact that the maximal compact subgroup $\SO_2(\R)=\{ \exp(s\Theta):0\leq s\leq 2\pi  \}$ acts transitively on every unit tangent space $T^1_{x+iy}\Hyp$, $x\in \R,\;y\in \R_{>0}$; this can be readily ascertained by means of the explicit expression in~\eqref{eq:isometricaction} for the isometric action of $\SL_2(\R)$ on $\Hyp$.
	
As a result, we can describe the geodesic circle $\mathcal{S}(z,t)$, defined as in~\eqref{eq:geodesiccircle} for any $z\in M$ and $t>0$, as the projection to $M$ of the translated curve
	\begin{equation}
		\label{eq:circlesbundle}
		\{\phi_t\circ \phi_s^{\Theta}(q): 0\leq s\leq \pi \}\;,
	\end{equation}
where $q$ is any preimage of $z$ under the fibration $\mathbf{p}$. In the latter curve, each point $z'$ in the geodesic circle $\mathcal{S}(z,t)$ appears as attached to its unique tangent vector identified by the derivative of the geodesic connecting $z$ with $z'$.

We are thus lead to consider the asymptotic distribution properties, inside the unit tangent bundle $Y$, of the translated curves in~\eqref{eq:circlesbundle} as $t$ tends to infinity. Any equidistribution statement about those admits then a clear-cut transposition, via the projection map $\mathbf{p}$, to an equidistribution statement concerning expanding geodesic circles on the surface $M$. The natural measure to consider on $\phi_t(\{\phi_s^{\Theta}(q):0\leq s\leq \pi  \})$ is the pushforward via $\phi_t$ of the uniform probability measure defined weakly by
\begin{equation*}
	f\mapsto \frac{1}{\pi}\int_{0}^{\pi}f\circ \phi_s^{\Theta}(q)\;\text{d}s\;, \quad f\colon Y\to \C\; \text{ continuous}.
\end{equation*}	
Notice that any limit of such measures, in the weak$^*$ topology, projects to the normalized hyperbolic area measure on $M$, in light of our prior knowledge of equidistribution of expanding geodesic circles on $M$ (cf.~Section~\ref{sec:hyperbolic}). Without appealing to the latter, we shall exhibit quantitatively that there is a unique weak$^*$ limit at the level of the unit tangent bundle, which is given by the uniform measure $m_Y$. 

\subsection{Asymptotics of averages along geodesic translates of homogeneous curves}
\label{subsec:asymptotics}

More generally, we consider averages along expanding homogeneous curves, that is, of the form 
\begin{equation*}
	\frac{1}{\sigma}\int_0^{\sigma}f\circ \phi_t\circ \phi_s^{W}(p)\;\text{d}s\;, \quad  W\in \sl_2(\R)\setminus\{0\},\;\sigma\in \R_{>0}, \;
\end{equation*}
for sufficiently regular observables $f\colon Y\to \C$.  Before proceeding with the general statement, some more notational and terminological preparation is required. We refer again to~\cite[Sec.~2]{Corso-Ravotti} for the precise definitions of all the notions and the facts we presently resort to. For every $s\in \R_{>0}$, let $W^{s}(Y)$ be the Sobolev space of order $s$ on $Y$; we write $\norm{f}_{W^s}$ for the Sobolev norm of a function $f\in W^s(Y)$. If $\mathscr{C}^{r}(Y)$ denotes, for every integer $r\geq 0$, the space of functions $f\colon Y\to \C$ of class $\mathscr{C}^{r}$, equipped with the $\mathscr{C}^{r}$-norm $\norm{\cdot}_{\mathscr{C}^{r}},$ then the Sobolev Embedding Theorem affirms that $W^{s}(Y)$ embeds continuously inside $\mathscr{C}^{r}(Y)$ whenever $s>r+3/2$. Whenever $s>3/2$, we shall in particular identify each function $f\in W^{s}(Y)$ with its unique continuous representative. For any bounded function $D\colon Y\to \C$, we denote by $\norm{D}_{\infty}$ its supremum norm.

Recall that $\Delta_M$ is the hyperbolic Laplace-Beltrami operator on $M$, with discrete pure-point spectrum $\text{Spec}(\Delta_M)\subset \R_{\geq 0}$. For every Laplace eigenvalue $\lambda$, let $r_{\lambda}$ be the unique complex number in $ \R_{\geq 0}\cup i \R_{>0}$ satisfying\footnote{The prominent role played by the value $1/4$ in the spectral geometry of hyperbolic surfaces is well-known: it identifies the lowest Laplace eigenvalue on the universal cover $\Hyp$ (see Bergeron's book~\cite{Bergeron} or Sarnak's survey~\cite{Sarnak-survey} for the spectral theory of hyperbolic surfaces). For our purposes, its presence in the spectrum $\text{Spec}(\Delta_M)$ affects the asymptotic expansion we are seeking after significantly, as shall shortly become apparent.} $\frac{1}{4}+r_{\lambda}^2=\lambda$. 
	
	We fix a basis of the Lie algebra $\sl_2(\R)$ given by the elements
	\begin{equation}
		\label{eq:UV}
		X=\begin{pmatrix} 1/2&0\\0&-1/2 \end{pmatrix}\;, \quad
		U=\begin{pmatrix} 0&1\\0&0 \end{pmatrix}\;, \quad V=\begin{pmatrix} 0&0\\1&0 \end{pmatrix}.
	\end{equation}
	
	Finally, let 
	\begin{equation*}
		\chi=
		\begin{cases}
			1 & \text{ if }1/4\in \text{Spec}(\Delta_M)\;,\\
			0 & \text{ if }1/4\notin \text{Spec}(\Delta_M)\;.
		\end{cases}
	\end{equation*}
	
	We are now in a position to phrase the main result of this section, which extends the asymptotics for expanding circle arcs\footnote{
		A midly weaker formulation of the same equidistribution result for expanding circle arcs had already been provided by Edwards in the unpublished manuscript~\cite{Edwards-unpublished} by entirely analogous methods. The authors of~\cite{Corso-Ravotti} were made aware of this earlier result after the completion of a first draft of~\cite{Corso-Ravotti}. } ~\cite[Thm.~1.8]{Corso-Ravotti} to geodesic translates of arbitrary homogeneous curves.
	
	\begin{thm}
		\label{thm:expandingtranslates}
		Let $\Ga<\SL_2(\R)$ be a cocompact lattice, $Y=\Ga\bsl \SL_2(\R)$, $a,b,c\in \R$ with $b\neq 0$, $W=aX+bU+cV\in \sl_2(\R)$. There exists $C\in \R_{>0}$, depending only on $\Ga $ and $W$, such that the following holds. Let $\sigma>0$, $s>11/2$ and $f\in W^{s}(Y)$; then there exist, for every positive eigenvalue $\lambda$ of the Laplace-Beltrami operator $\Delta_M$ on $M=\Ga\bsl \Hyp$, continuous functions $D^{+}_{W,\sigma,\lambda}f,D^{-}_{W,\sigma,\lambda}f\colon Y\to \C$ with
		\begin{equation*}
			\sum_{\lambda\in \emph{Spec}(\Delta_M)\setminus\{0\}}\norm{D^{+}_{W,\sigma,\lambda}f}_{\infty}+\norm{D^{-}_{W,\sigma,\lambda}f}_{\infty}\leq \frac{C}{\sigma}\norm{f}_{W^s}
		\end{equation*}
		such that, for every $t$ sufficiently large, depending on the previous parameters, and every $p \in Y$,
		\begin{equation}
			\label{eq:asymptotics}
			\begin{split}
				\frac{1}{\sigma}\int_{0}^{\sigma}f\circ \phi_{-t}\circ \phi_\xi^{W}(p)\;\emph{d}\xi=&\int_{Y} f\;\emph{d}m_Y\\
				&+e^{-t/2}\sum_{\lambda \in \emph{Spec}(\Delta_M),\;\lambda>1/4}\cos{(r_{\lambda}t)}D^{+}_{W,\sigma,\lambda}f(p)+\sin{(r_{\lambda}t)}D^{-}_{W,\sigma,\lambda}f(p)\\
				&+\sum_{\lambda \in \emph{Spec}(\Delta_M),\;0<\lambda<1/4}e^{-\bigl(\frac{1}{2}-ir_{\lambda}\bigr)t}D^{+}_{W,\sigma,\lambda}f(p)+e^{-\bigl(\frac{1}{2}+ir_{\lambda}\bigr)t}D^{-}_{W,\sigma,\lambda}f(p)\\
				&+\chi\biggl(e^{-\frac{t}{2}}D^{+}_{W,\sigma,1/4}f(p)+te^{-\frac{t}{2}}D^{-}_{W,\sigma,1/4}f(p)\biggr)\\
				&+\mathcal{R}_{W,\sigma}f(p,t)\;,
			\end{split}
		\end{equation}
		where
		\begin{equation*}
			|\mathcal{R}_{W,\sigma}f(p,t)|\leq \frac{C}{\sigma}\norm{f}_{W^s}(t+1)e^{-t}\;.
		\end{equation*}
	\end{thm}

As far as the validity of the statement is concerned, the choice of negative values of $t$, only made on convenience grounds in the upcoming arguments, is immaterial: an entirely analogous asymptotic expansion holds for positive values of $t$, provided that the assumption $b\neq 0$ is replaced by $c\neq 0$. Glossing over the technical, straightforward adaptations, we only provide here a heuristic explanation.
The one-parameter flows on $Y$ generated by $U$ and $V$ are, respectively, the stable and unstable horocycle flow (see, for instance,~\cite[Chap.~11]{Einsiedler-Ward} or~\cite[Chap.~4]{Bekka-Mayer}), which parametrize stable and unstable leaves for the geodesic flow $(\phi_t)_{t\in \R}$. The assumption $b\neq 0$ in Theorem~\ref{thm:expandingtranslates} amounts thus to the presence of a non-zero component in the stable direction, which is expanded for negative times of $t$, for the vector $W$ dictating the direction of homogeneous curve we start with. As such, it is indispensable for equidistribution to hold even just from a qualitative standpoint. 

Recall that, in order to approach the equidistribution problem for expanding dilates of subsets of Euclidean spaces in Section~\ref{sec:qualitativetorus}, we decomposed the test function $f$ as a uniformly converging sum of its projections onto eigenspaces of the Laplace-Beltrami operator on the $d$-dimensional torus $\T^d$; these eigenspaces are indeed generated by unitary characters, as is well known. As outlined in more detail in Section~\ref{sec:proof}, we shall proceed here in a similar way; in particular, the various summands appearing in the uniformly converging infinite sum on the right-hand side of~\eqref{eq:asymptotics}, indexed by Laplace eigenvalues, represent the leading terms of the asymptotic expansion for the projection of $f$ onto the corresponding Laplace (or, more precisely as clarified below, Casimir) eigenspace. The remainder term $\mathcal{R}_{W,\sigma}f(p,t)$ comprises the contributions of all the relative lower order terms, together with the potentially non-vanishing terms coming from the discrete-series components of $f$ (that is, those corresponding to negative Casimir eigenvalues, see Section~\ref{sec:proof}); the latter are absent when $f$ is the pullback to $Y$ of a function defined on the underlying surface $M$.

	\begin{rmk}
		We collect here a few more comments about Theorem~\ref{thm:expandingtranslates}.	
		\begin{itemize}
		\item[(a)] If $f$ is defined on the surface $M$, that is, if there exists $\tilde{f}\in W^{s}(M)$ such that $f=\tilde{f}\circ \mathbf{p}$, then it suffices that $s$ be larger than $9/2$ for the statement to hold. This is elaborated upon in~\cite[Sec.~5.4]{Corso-Ravotti} for the special case $W=\Theta$; those considerations apply here as well.
		\item[(c)] From the geometric intuition, it is expected that the equidistribution rate improves as the length of the original arc increases. The statement quantifies this, giving an estimate of order $\theta^{-1}$ both in the uniform norms of the coefficients appearing in the leading terms of the asymptotic expansion~\eqref{eq:asymptotics}, and in the error term.
		\item[(c)] The asymptotic expansion in~\eqref{eq:asymptotics} allows to establish limit theorems concerning the statistical behaviour, when the base point $p$ is sampled according to a given probability distribution on $Y$, of the deviations of the averages
		\begin{equation*}
			\frac{1}{\sigma}\int_0^{\sigma}f\circ \phi_{-t}\circ \phi_s^{W}(p)\;\text{d}s
		\end{equation*}
		 from their limiting value $\int_{Y}f\;\text{d}m_Y$. Remarkably, the limiting distribution governing the asymptotic behaviour of an appropriate rescaling of such deviations is not Gaussian; in contrast, it happens to be compactly supported, being dictated by the uniformly bounded coefficients $D^{\pm}_{W,\sigma,\lambda}f$.
		 These and related distributional aspects are discussed at length in~\cite{Corso-Ravotti}, to which we refer for precise statements and complete proofs. 
		 \item[(d)] In the case of expanding segments of horocycle orbits, that is, when $W=U$, the statement is already present implicitly in~\cite[Thm.~2]{Rav}, for the familiar commutation relation 
		 \begin{equation*}
		 	\phi_t\circ \phi_s^{U}(p)=\phi_{e^{-t}s}^{U}(p)\circ \phi_t(p)\;, \quad s,t\in \R, \; p\in Y
		 \end{equation*}
	 between the geodesic and the horocycle flow allows to express the horocycle ergodic averages in~\cite[Thm.~2]{Rav} as appropriately rescaled geodesic translates of segments of horocycle orbits with a moving base point. 
	\end{itemize}
\end{rmk}

\pagebreak
	We now make explicit the resulting optimal effective equidistribution statement, which follows at once from Theorem~\ref{thm:expandingtranslates}. Let $\lambda_*=\lambda_1$ be the spectral gap of $M$, $r_*$ the corresponding parameter as above.

	\begin{cor}
		Notation is as in Theorem~\ref{thm:expandingtranslates}. There exists a function $D^{\emph{main}}_{W,\sigma}f\colon Y\times \R_{>0} \to \C$
		with
		\begin{equation*}
			\sup_{p\in Y,\; t>0}|D^{\emph{main}}_{W,\sigma}(p,t)|\leq \frac{C}{\sigma}\norm{f}_{W^s}
		\end{equation*}
		 such that, for every $t$ sufficiently large and every $p\in Y$,
		\begin{equation}
			\label{eq:effective}
			\frac{1}{\sigma}\int_0^{\sigma}f\circ \phi_{-t}\circ \phi_s^{W}(p)\;\emph{d}s=\int_{Y}f\;\emph{d}m_Y+ t^{\chi}e^{-\bigl(\frac{1}{2}+\Im(r_*)\bigr)t}D^{\emph{main}}_{W,\sigma}f(p,t)+o\biggl(e^{-\bigl(\frac{1}{2}+\Im(r_*)\bigr)t}\biggr)  \;.
		\end{equation}
	
	\end{cor}

It is well known that the spectral gap $\lambda_*$ encodes the geometric features of the hyperbolic surface $M$ in a number of different ways; see, for instance,~\cite{Berger,Bergeron,Sarnak}. From the corollary we observe that it governs the optimal equidistribution rate of geodesic translates of homogeneous curves. It is worth noticing that the latter rate matches exactly the mixing rate of the geodesic flow\footnote{This feature of exact matching of the exponents for the rates of mixing and of equidistribution of expanding translates is by no means a general property of mixing homogeneous flows. The horocycle flow $(\phi_t^{U})_{t\in \R}$ on $Y$ is a counterexample, as observed by Ravotti in~\cite{Ravotti-arcs}; in particular, compare~\cite[Thm.~1]{Ravotti-arcs} with~\cite[Cor.~3]{Ravotti-arcs}.} on $Y$ obtained by Ratner in~\cite{Ratner}.  

\begin{rmk}
	\begin{itemize}
	\item[(a)] As pointed out in~\cite[Rmk.~1.9]{Corso-Ravotti}, to which we refer for the details, it is possible to deduce an essentially equivalent version of Theorem~\ref{thm:expandingtranslates} from the work of Bufetov and Forni in~\cite{Bufetov-Forni}.
	\item[(b)] As the full expansion in~\eqref{eq:asymptotics} clearly shows, the main coefficient $D^{\text{main}}_{W,\sigma}f$ becomes only a function of the base point $p$ as soon as $M$ has small eigenvalues, that is, whenever $\lambda_*\leq 1/4$.
	\item[(b)] In the vein of Strömbergsson's results in~\cite{Strombergsson-closedhorocycles}, it is possible to make use of the asymptotics in~\eqref{eq:effective} to establish effective equidistribution of shrinking pieces of homogeneous curves; in other words, effective equidistribution is retained when the length parameter $\sigma$ is allowed to decrease with $t$, with a speed that can be exponential but compatible with the highest-order term $t^{\chi}e^{-\bigl(\frac{1}{2}+\Im(r_*)\bigr)t}$ of the asymptotics. For a detailed argument, see the proof of~\cite[Cor.~1.14]{Corso-Ravotti}, dealing with the special case $W=\Theta$; everything there extends immediately to our current setting. 
	\end{itemize}
\end{rmk}

Before turning to an outline of the proof of Theorem~\ref{thm:expandingtranslates}, we make a final comment on the case of expanding circles, obtained here by specializing the theorem to $W=\Theta$. In the Euclidean case (cf.~Corollary~\ref{cor:dilatingsurfaces}), the question arises as to whether effective equidistribution of dilated circles can be similarly upgraded to unit tangent bundles. The answer is affirmative, and the equidistribution rate on the unit tangent bundle $T^1\T^2$ matches the one on the base $\T^2$; we relegate the proof of this fact to Appendix~\ref{secapp:app}.

\section{Elements of the proof of the main result}
\label{sec:proof}
	
	We shall confine the exposition to those aspects of the proof of Theorem~\ref{thm:expandingtranslates} which, albeit not substantially, differ from the argument leading to~\cite[Thm.~1.8]{Corso-Ravotti}, that is, to the special case $W=\Theta$, where $\Theta$ has been defined in~\eqref{eq:theta}. For the remaining steps of the proof, which carry over unaffectedly to a general $W=aX+bU+cV$ subject to the condition $b\neq 0$,  we refer to~\cite{Corso-Ravotti}. Here $X$ is defined as in~\eqref{eq:X}, and $U,V$ as in~\eqref{eq:UV}.
	
	Let us thus fix a cocompact lattice $\Ga<\SL_2(\R)$ and a left-invariant vector field $W\in \sl_2(\R)$ as above. Denote the quotient $\Ga\bsl \SL_2(\R)$ by $Y$, and equip it with the unique $\SL_2(\R)$-invariant probability measure $m_Y$. We shall regard every element $Z\in \sl_2(\R)$ as a smooth vector field on $Y$, namely as the infinitesimal generator of the smooth flow $(\phi_s^{Z})_{s\in \R}$ which has been defined in~\eqref{eq:homogeneousflow}. This extends to an identification of every element in the universal enveloping algebra $U(\sl_2(\R))$ of $\sl_2(\R)$ with a differential operator acting on smooth functions on $Y$.
	
	The measure-preserving action of $\SL_2(\R)$ on the probability space $(Y,m_Y)$ gives rise to a unitary representation $\rho$ of $\SL_2(\R)$ onto the complex Hilbert space $\mathscr{H}=L^{2}(Y,m_Y)$.
	Define the Casimir operator $\square$ as the second-order linear differential operator $\square=-X^2+X-UV$. It is a generator of the center of $U(\sl_2(\R))$, and as such it governs the structure of the representation $\rho$ of $\SL_2(\R)$ in the following sense. Regarding $\square$ as an essentially self-adjoint unbounded operator on $\mathscr{H}$, its spectrum $\text{Spec}(\square)$ comprises the Laplace spectrum\footnote{This is due to the fact that, with the chosen normalization, the Casimir operator coincides with the Laplace-Beltrami operator $\Delta_M$ when acting on smooth functions defined on the underlying surface $M$.}  $\text{Spec}(\Delta_M)$ and a fully understood discrete set of negative eigenvalues. By compactness of $Y$, the elementary theory of unitary representations of $\SL_2(\R)$ (see~\cite{Bargmann,Lang-sl}) gives that the representation space $\mathscr{H}$ decomposes as an orthogonal direct sum of eigenspaces of the operator $\square$:
	\begin{equation}
		\label{eq:firstorthogonal}
		\mathscr{H}=\bigoplus\limits_{ \mu \in \text{Spec}(\square)} \mathscr{H}_{\mu}\:, \quad \mathscr{H}_{\mu}=\overline{\{v\in \mathscr{C}^{\infty}(Y):\square v=\mu v \}}^{\mathscr{H}}\;,
	\end{equation}
where $\overline{V}^{\mathscr{H}}$ denotes the closure inside $\mathscr{H}$ of a linear subspace $V\leq \mathscr{H}$. What is more, by an elementary feature of the unitary representation theory of the compact abelian group $\SO_2(\R)$, each eigenspace $\mathscr{H}_{\mu}$ decomposes further into $\SO_2(\R)$-invariant subspaces: for every Casimir eigenvalue $\mu$, there exists a subset $I(\mu)\subset \Z$ such that 
\begin{equation}
	\label{eq:secondorthogonal}
	\mathscr{H}_{\mu}=\bigoplus_{n\in I(\mu)}\mathscr{H}_{\mu,n}\;, \quad \mathscr{H}_{\mu,n}=\overline{\{ v\in \mathscr{C}^{\infty}(Y)\cap \mathscr{H}_\mu: \Theta v= in v \}}^{\mathscr{H}}\;.
\end{equation}
The orthogonal decompositions in~\eqref{eq:firstorthogonal} and~\eqref{eq:secondorthogonal} hold at the level of Sobolev spaces as well: for every $s>0$, 
\begin{equation*}
	W^{s}(Y)=W^{s}(\mathscr{H})=\bigoplus\limits_{\mu \in \text{Spec}(\square)} \bigoplus_{n\in I(\mu)}W^{s}(\mathscr{H}_{\mu,n})\;.
\end{equation*}
Therefore,	if $f$ is a test function in the Sobolev space $W^{s}(Y)$,  then it admits a decomposition
	\begin{equation}
		\label{eq:decomp}
		f=\sum_{n\in \N}\sum_{n\in I(\mu)}f_{\mu,n}\;,
	\end{equation}
where the sum converges in the $W^{s}$-norm and $f_{\mu,n}$ is the orthogonal projection of $f$ onto the closed subspace $W^{s}(\mathscr{H}_{\mu,n})$, for every Casimir eigenvalue $\mu$ and every $n\in I(\mu)$. In light of the Sobolev Embedding Theorem, we can ensure that $f_{\mu,n}$ is of class $\mathscr{C}^{2}$ on $Y$ by taking $s$ sufficiently large.

The gist of the argument resides thus in the derivation of an asymptotics for 
\begin{equation*}
	\frac{1}{\sigma}\int_0^{\sigma}g\circ \phi_{-t}\circ \phi_s^{W}(p)\;\text{d}s
\end{equation*}
when $g\in \mathscr{C}^{2}(Y)$ is a joint eigenfunction of the operators $\square$ and $\Theta$. By means of the decomposition in~\eqref{eq:decomp}, we can then add up the contributions to the asymptotic expansion coming from the various components $f_{\mu,n}$, with the caveat that $s$ should be chosen large enough for all the involved infinite sums to converge absolutely. This is the reason underlying the assumption $s>11/2$ in Theorem~\ref{thm:expandingtranslates}, which is thoroughly discussed in~\cite[Sec.~5]{Corso-Ravotti}.

\medskip
On account of the previous discussion, we have reduced matters to the case of a test function $f\in \mathscr{C}^{2}(Y)$ satisfying 
\begin{equation}
	\label{eq:assm}
	\square{f}=\mu f\;, \quad \Theta f=in f
\end{equation}
for some $\mu \in \text{Spec}(\square)$ and $n\in I(\mu)$. Let us also fix  a base point $p\in Y$ and a length parameter $\sigma \in \R_{>0}$. 

In view of the assumption~\eqref{eq:assm}, it shall be convenient to work with the basis $\{X,\Theta,R  \}$ of the real vector space $\sl_2(\R)$, where
	\begin{equation*}
	 R=\begin{pmatrix}  0&1\\
		1&0
		\end{pmatrix}
	\;.
	\end{equation*}

	We write thus 
	\begin{equation*}
	W=\alpha X +\beta \Theta+\gamma R
	\end{equation*}
	 for some $\alpha,\beta,\gamma\in \R$. For ease of reading, let us set
	\begin{equation}
		\label{eq:k}
		k(t)=\frac{1}{\sigma}\int_0^{\sigma}f\circ \phi_{-t}\circ \phi_s^{W}(p)\;\text{d}s\;, \quad t\in \R.
	\end{equation}

Since $X$ is the infinitesimal generator of the geodesic flow $(\phi_t)_{t\in \R}$,  differentiation under the integral sign readily gives
\begin{equation}
	\label{eq:kprime}
	k'(t)=\frac{1}{\sigma}\int_0^{\sigma}-Xf\circ \phi_{-t}\circ \phi^{W}_s(p)\;\text{d}s\;, \quad k''(t)=\frac{1}{\sigma}\int_0^{\sigma}X^2 f \circ \phi_{-t}\circ \phi^{W}_s(p)\;\text{d}s
\end{equation}
for every $t\in \R$.
As a result, the partial differential equation
\begin{equation*}
	-X^{2}f+Xf-UV f=\mu f\;,
\end{equation*}
encoding the fact that $f$ is a Casimir eigenfunction, results into the ordinary differential equation 
\begin{equation*}
	k''(t)+k'(t)+\mu k(t)=\frac{1}{\sigma}\int_0^{\sigma}UV f \circ \phi_{-t}\circ \phi^{W}_s(p)\;\text{d}s
\end{equation*}
fulfilled by $k$. As $V=U-\Theta$, and recalling that $\Theta f=in f$, we get that
\begin{equation}
	\label{eq:firstODE}
	k''(t)+k'(t)+\mu k(t)=\frac{1}{\sigma}\int_0^{\sigma}U^{2}f\circ \phi_{-t}\circ \phi_s^{W}(p)\;\text{d}s-\frac{in}{\sigma}\int_0^{\sigma}Uf\circ \phi_{-t}\circ \phi^{W}_s(p)\;\text{d}s\;.
\end{equation}

We now perform some elementary manipulations to obtain a more explicit expression for the right-hand side of~\eqref{eq:firstODE}. 

\begin{lem}
	\label{lem:UandUsquare}
Suppose given $f\in \mathscr{C}^2(Y)$, $p\in Y$ and $\sigma>0$, and	define three functions $A,B,C\colon \R\to \C$ by
	\begin{align*}
		&A(t)=\frac{1}{\sigma}\bigl(f\circ\phi_{-t}\circ \phi_{\sigma}^{W}(p)-f\circ \phi_{-t}(p) \bigr) \\
		& B(t)=\frac{1}{\sigma}\bigl(Uf\circ\phi_{-t}\circ \phi_{\sigma}^{W}(p)-Uf\circ \phi_{-t}(p) \bigr)\\
		&C(t)=\frac{1}{\sigma}\bigl(Xf\circ\phi_{-t}\circ \phi_{\sigma}^{W}(p)-Xf\circ \phi_{-t}(p) \bigr) \;.
	\end{align*}
Assume that $\Theta f=in f$ for some $n\in \Z$. Then we have
	\begin{equation}
		\label{eq:Uf}
		\frac{1}{\sigma}\int_0^{\sigma}Uf\circ \phi_{-t}\circ \phi^{W}_s(p)\;\emph{d}s= \frac{e^{-t}}{\gamma+\beta +(\gamma-\beta)e^{-2t}}\bigl(A(t)+\alpha k'(t)+in (\gamma-\beta)e^{-t}k(t)\bigr)
	\end{equation}
and
\begin{equation}
	\label{eq:Usquaref}
	\begin{split}
	\frac{1}{\sigma}&\int_0^{\sigma}U^2f\circ \phi_{-t}\circ \phi^{W}_s(p)\;\emph{d}s=\\  &=\frac{e^{-t}}{\gamma+\beta +(\gamma-\beta)e^{-2t}}
	\biggl(B(t)
	- \frac{e^{-t}}{\gamma+\beta +(\gamma-\beta)e^{-2t}}  \bigl(\alpha(C(t)-\alpha k''(t)-in(\gamma-\beta)e^{-t}k''(t))\\
	&\quad +(\alpha-in(\gamma-\beta)e^{-t})(A(t)+\alpha k'(t)+in (\gamma-\beta)e^{-t}k(t))\bigr)
	\biggr)\;.
	\end{split}
	\end{equation}
for every $t\in \R$ for which the right-hand sides of~\eqref{eq:Uf} and~\eqref{eq:Usquaref} are defined.
\end{lem}

Before embarking on the proof of Lemma~\ref{lem:UandUsquare}, we first introduce some notation. If $\gamma\colon \R\to Y$ is a smooth curve, we indicate with $\frac{\text{d}}{{\text{d}s}}\gamma(s)$ its derivative at time $s\in \R$, which is an element of the tangent space $T_{\gamma(s)}Y$. If $f\colon Y\to \C$ is a smooth function, $\text{d}f_q\colon T_qY\to \C$ is its differential at a point $q\in Y$. If $Z$ is a smooth vector field on $Y$, we denote by $Z_q$ its value at the point $q\in Y$.  For every $g\in \SL_2(\R)$, we denote by $\Ad_g\colon \sl_2(\R\to \sl_2(\R)$ the adjoint action of $g$ on $\sl_2(\R)$, given explicitly by $\Ad_g(x)=gxg^{-1}$ for every $x\in \sl_2(\R)$. Finally, recall that $\exp\colon \sl_2(\R)\to \SL_2(\R)$ denotes the exponential map. 

The proof of Lemma~\ref{lem:UandUsquare}  relies on the following elementary fact about derivatives of translated homogeneous curves, a proof of which is given in~\cite[Lem.~5]{Ravotti-arcs}.

\begin{lem}
	\label{lem:derivativetranslates}
	Let $Z_1,Z_2\in \sl_2(\R)$, $p\in Y$. Then, for every $s\in \R$, 
	\begin{equation*}
		\frac{\emph{d}}{\emph{ds}}(\phi^{Z_1}_t\circ \phi^{Z_2}_s(p))=\Ad_{\exp{(-tZ_1)}}(Z_2)_{\phi^{Z_1}_t\circ \phi_s^{Z_2}(p)}\;.
	\end{equation*}
\end{lem}

\medskip

\begin{proof}[Proof of Lemma~\ref{lem:UandUsquare}]
	We fix $f, p, \sigma$ and define $A,B,C$ as indicated in the statement. Recall also that we express $W=\alpha X +\beta \Theta+\gamma R$. Stokes' theorem gives
	\begin{equation}
		\label{eq:one}
		A(t)=\frac{1}{\sigma}\int_0^{\sigma}\frac{\text{d}}{\text{d}s}(f\circ \phi_{-t}\circ \phi^{W}_s(p))\;\text{d}s\;.
	\end{equation}
The chain rule for differentiation allows to write the derivative appearing in the integrand of the last expression as
\begin{equation}
	\label{eq:two}
	\begin{split}
	\frac{\text{d}}{\text{d}s}(f\circ \phi_{-t}\circ \phi^{W}_s(p))&=\text{d}f_{\phi_{-t}\circ \phi_s^{W}(p)}\biggl(\frac{\text{d}}{\text{d}s}\phi_{-t}\circ \phi_s^{W}(p)\biggr)=\text{d}f_{\phi_{-t}\circ \phi_s^{W}(p)}\biggl(\Ad_{\exp{(tX)}}(W)_{\phi_{-t}\circ \phi_s^{W}(p)}\biggr)\;,
	\end{split}
\end{equation}
the last equality being given by Lemma~\ref{lem:derivativetranslates}. A straightforward matrix computation shows that 
\begin{equation}
	\label{eq:three}
	\begin{split}
	\Ad_{\exp(tX)}(W)&=\exp(tX)W\exp(-tX)=\alpha X + (\gamma+\beta)e^{t}U+(\gamma-\beta)e^{-t}V\\&=\alpha X + ((\gamma+\beta)e^{t}+(\gamma-\beta)e^{-t})U-(\gamma-\beta)e^{-t}\Theta\;.
	\end{split}
\end{equation}
Combining~\eqref{eq:one},~\eqref{eq:two} and~\eqref{eq:three} yields
\begin{equation*}
	\begin{split}
	A(t)=&\frac{\alpha}{\sigma}\int_0^{\sigma}Xf\circ \phi_{-t}\circ \phi_s^{W}(p)\;\text{d}s+\frac{(\gamma+\beta)e^{t}+(\gamma-\beta)e^{-t}}{\sigma}\int_0^{\sigma}Uf\circ \phi_{-t}\circ \phi_s^{W}(p)\;\text{d}s\\
	&-\frac{in(\gamma-\beta)e^{-t}}{\theta}\int_0^{\theta}f\circ \phi_{-t}\circ \phi_s^{W}(p)\;\text{d}s
	\end{split}
\end{equation*}
which, recalling~\eqref{eq:k} and~\eqref{eq:kprime},  can be rewritten as 
\begin{equation*}
	A(t)=-\alpha k'(t)-in(\gamma-\beta)e^{-t}k(t)+\frac{(\gamma+\beta)e^{t}+(\gamma-\beta)e^{-t}}{\sigma}\int_0^{\sigma}Uf\circ \phi_{-t}\circ \phi_s^{W}(p)\;\text{d}s\;.
\end{equation*}
This establishes the equality in~\eqref{eq:Uf}.

 Starting over again with the function $Uf$ in place of $f$, we get
\begin{equation}
	\label{eq:Bt}
	\begin{split}
	B(t)&=\frac{1}{\sigma}\int_0^{\sigma}\frac{\text{d}}{\text{d}s}(Uf\circ \phi_{-t}\circ \phi_s^{W}(p))\;\text{d}s=\frac{1}{\sigma}\int_0^{\sigma}\text{d}(Uf)_{\phi_{-t}\circ\phi_s^{W}(p)}(\Ad_{\exp{(tX)}}(W)_{\phi_{-t}\circ \phi_s^{W}(p)})\;\text{d}s\\
	&=\frac{\alpha}{\sigma}\int_0^{\sigma}XUf\circ \phi_{-t}\circ \phi_s^{W}(p)\;\text{d}s-\frac{(\gamma-\beta)e^{-t}}{\sigma}\int_0^{\sigma}\Theta Uf\circ \phi_{-t}\circ \phi_s^{W}(p)\;\text{d}s\\
	&+\frac{(\gamma+\beta)e^{t}+(\gamma-\beta)e^{-t}}{\sigma}\int_0^{\sigma}U^{2}f\circ \phi_{-t}\circ \phi_s^{W}(p)\;\text{d}s\;.
	\end{split}
\end{equation}

\pagebreak

We write $[\cdot,\cdot]$ for the Lie bracket in $\sl_2(\R)$. Because of the commutator relations
\begin{equation*}
	[X,U]=U\;, \quad [\Theta,U]=2X\;,
\end{equation*}
which are straightforward to verify, we can write
\begin{equation}
	\label{eq:firstterm}
	\int_0^{\sigma} XUf\circ \phi_{-t}\circ \phi_s^{W}(p)\;\text{d}s=\int_0^{\sigma} U Xf\circ \phi_{-t}\circ \phi_s^{W}(p)\;\text{d}s+\int_0^{\sigma} Uf\circ \phi_{-t}\circ \phi_s^{W}(p)\;\text{d}s
\end{equation}
and
\begin{equation}
	\label{eq:secondterm}
	\begin{split}
	\int_0^{\sigma}\Theta Uf\circ \phi_{-t}\circ \phi_s^{W}(p)\;\text{d}s&=\int_0^{\sigma} U\Theta f\circ \phi_{-t}\circ \phi_s^{W}(p)\;\text{d}s+2\int_0^{\sigma}Xf\circ \phi_{-t}\circ \phi_s^{W}(p)\;\text{d}s\\
	&= in\int_0^{\sigma} Uf\circ \phi_{-t}\circ \phi_s^{W}(p)\;\text{d}s-2k'(t)\;.
	\end{split}
\end{equation}

Making use of the already established expression in~\eqref{eq:Uf} for $\frac{1}{\sigma}\int_0^{\sigma}Ug\circ \phi_{-t}\circ \phi_s^{W}(p)\;\text{d}s$, when $g\in \{f,Xf  \}$, and combining~\eqref{eq:Bt},~\eqref{eq:firstterm} and~\eqref{eq:secondterm}, we get
\begin{equation*}
	\begin{split}
	B(t)&=\frac{\alpha}{\sigma}\int_0^{\sigma} UXf\circ \phi_{-t}\circ \phi_s^{W}(p)\;\text{d}s +2(\gamma-\beta)e^{-t}k'(t) \\
	&\quad +
	 \frac{\alpha-in (\gamma-\beta)e^{-t}}{\sigma}\int_0^{\sigma} Uf\circ \phi_{-t}\circ \phi_s^{W}(p)\;\text{d}s\\
	 & \quad +\frac{(\gamma+\beta)e^{t}+(\gamma-\beta)e^{-t}}{\sigma}\int_0^{\sigma}U^{2}f\circ \phi_{-t}\circ \phi_s^{W}(p)\;\text{d}s\\
	 &=\frac{\alpha e^{-t}}{\gamma+\beta+(\gamma-\beta)e^{-2t}}\bigl(C(t)-\alpha k''(t)-in(\gamma-\beta)e^{-t}k'(t)\bigr)\\
	 & \quad + \frac{e^{-t} (\alpha-in (\gamma-\beta)e^{-t})}{\gamma+\beta +(\gamma-\beta)e^{-2t}}\bigl(A(t)+\alpha k'(t)+in (\gamma-\beta)e^{-t}k(t)\bigr)\\
	 & \quad +\frac{(\gamma+\beta)e^{t}
	 	+(\gamma-\beta)e^{-t}}{\sigma}\int_0^{\sigma}U^{2}f\circ \phi_{-t}\circ \phi_s^{W}(p)\;\text{d}s\;,
	\end{split}
\end{equation*}
from which the equality in~\eqref{eq:Usquaref} follows at once.
\end{proof}

We can now gather the information obtained thus far in order to reach an expression for the ordinary differential equation in~\eqref{eq:firstODE} which is amenable to an explicit analytic investigation. 

\begin{lem}
	Let $W=\alpha X+\beta \Theta+\gamma R\in \sl_2(\R)$ with $\gamma\neq -\beta$. Suppose that $f\in \mathscr{C}^{2}(Y)$ satisfies $\square{f}=\mu f$ and $\Theta f =inf$ for some $\mu\in \emph{Spec}(\square)$ and $n\in \Z$. Let $p\in Y$ and $\sigma>0$. Then there exist $t_0=t_0(W)>0$ and a continuous bounded function $G\colon [t_0,+\infty)\to \C$ such that the function $k\colon \R\to \C$ defined in~\eqref{eq:k} satisfies the linear second-order ordinary differential equation
	\begin{equation}
		\label{eq:finalODE}
		k''(t)+k'(t)+\mu k(t)=e^{-t}G(t)
	\end{equation}
for every $t>t_0$.
\end{lem}

Notice that the condition $\gamma \neq -\beta$ appearing in the foregoing statement is the transposition to the new basis $\{X,\Theta,R\}$ of the assumption $b\neq 0$ in the statement of Theorem~\ref{thm:expandingtranslates}.

\begin{proof}
The statement follows directly from the identity in~\eqref{eq:firstODE} by means of Lemma~\ref{lem:UandUsquare}: it suffices to set
	\begin{equation}
		\label{eq:G}
		\begin{split}
		G(t)=& \frac{1}{\gamma+\beta+(\gamma-\beta)e^{-2t}}\biggl(  B(t)
		- \frac{e^{-t}}{\gamma+\beta +(\gamma-\beta)e^{-2t}}
		 \bigl(\alpha(C(t)-\alpha k''(t)-in(\gamma-\beta)e^{-t}k''(t))\\
		 &+(\alpha-in(\gamma-\beta)e^{-t})(A(t)
		 +\alpha k'(t)+in (\gamma-\beta)e^{-t}k(t))\bigr) 
	\\
	&	-in \bigl(A(t)+\alpha k'(t)+in (\gamma-\beta)e^{-t}k(t)\bigr)\biggr)\;.
		\end{split}
	\end{equation}

\pagebreak

Recalling the definition of the functions $A,B,C$ in the statement of Lemma~\ref{lem:UandUsquare}, continuity of $G$ on its domain of definition is an immediate consequence of the fact that $f$ is of class $\mathscr{C}^{2}$ on $Y$.
Now observe that the terms $k(t),k'(t),k''(t),A(t),B(t)$ and $C(t)$ are all uniformly bounded, in absolute value, by the $\mathscr{C}^{1}$-norm of $f$. Furthermore, since we are assuming that $\gamma\neq \beta$, the factor
\begin{equation*}
	\frac{1}{\gamma+\beta+(\gamma-\beta)e^{-2t}}
\end{equation*}
 is uniformly bounded in absolute value for every $t$ larger than a certain threshold $t_0$ depending on $\beta$ and $\gamma$, and thus ultimately on the vector field $W$.
\end{proof}

The differential equation in~\eqref{eq:finalODE} can be explicitly solved in terms of the function $G(t)$.

\begin{lem}[{\cite[Lemma 4.1]{Corso-Ravotti}}]
	\label{lem:solution}
	Let $t_0\in \R_{>0}$, $G\colon (t_0,+\infty)\to \C$ a continuous bounded function. Suppose given $\mu\in \R$, and let $\nu$ be the unique complex number in $\R_{\geq 0}\cup i \R_{>0}$ satisfying\linebreak $1-\nu^{2}=4\mu$. Consider the ordinary differential equation
	\begin{equation}
		\label{eq:ODE}
		y''(t)+y'(t)+\mu y(t)=e^{-t}G(t)\;, \quad t>t_0
	\end{equation}
in the unknown $y(t)$. Fix also some real number $t_1>t_0$. Then, a function $k\colon (t_0,+\infty)\to \C$ of class $\mathscr{C}^{2}$ is a solution of~\eqref{eq:ODE} if and only if it takes one of the following two forms, depending on the value of $\mu$:
\begin{enumerate}
	\item when $\mu\neq 1/4$, 
\begin{equation}
	\label{eq:munotquarter}
	k(t)=e^{-\frac{1-\nu}{2}t}\biggl(c_1+\frac{1}{\nu}\int_{t_1}^{t}e^{-\frac{1+\nu}{2}\xi}G(\xi)\;\emph{d}\xi\biggr)+e^{-\frac{1+\nu}{2}t}\biggl(c_2-\frac{1}{\nu}\int_{t_1}^{t}e^{-\frac{1-\nu}{2}\xi}G(\xi)\emph{d}\xi\biggr)
\end{equation}
for some $c_1,c_2\in \C$;
\item when $\mu=1/4$,
\begin{equation}
	\label{eq:muquarter}
	k(t)=e^{-t/2}\biggl(c_1+\int_{t_1}^{t}\xi e^{-\xi/2}G(\xi)\;\emph{d}\xi\biggr)+te^{-t/2}\biggl(c_2+\int_{t_1}^{t}e^{-\xi/2}G(\xi)\emph{d}\xi\biggr)
\end{equation}
for some $c_1,c_2\in \C$.
\end{enumerate}
\end{lem}

Given the analytic formulas  in~\eqref{eq:munotquarter}  and~\eqref{eq:muquarter},  it is possible to derive an explicit asymptotic expansion for $k(t)$ as $t$ tends to infinity. We only outline the deduction in the case $0<\mu<1/4$, corresponding to $0<\nu<1$ for $\nu$ as in the statement of Lemma~\ref{lem:solution}, referring instead to\linebreak\cite[Sec.~4]{Corso-Ravotti} for a detailed treatment of all other cases. 

First, the values of $c_1$ and $c_2$ are uniquely determined by taking into account the initial conditions
\begin{equation*}
	k(t_1)=\frac{1}{\sigma}\int_0^{\sigma}f\circ \phi_{-t_1}\circ \phi_s^{W}(p)\;\text{d}s\;,\quad  k'(t_1)=-\frac{1}{\sigma}\int_0^{\sigma}Xf\circ \phi_{-t_1}\circ \phi_s^{W}(p)\;\text{d}s\;.
\end{equation*} 
As $G$ is bounded on $(t_0,+\infty)$, the functions $\xi\mapsto e^{-\frac{1\pm \nu}{2}\xi}G(\xi)$ are integrable on the closed half-line $[t_1,+\infty)$. As a consequence, we may write
\begin{equation*}
	\begin{split}
	k(t)=&e^{-\frac{1+\nu}{2}t}D^{+}_{W,\sigma,\mu,n}f(p)+e^{-\frac{1-\nu}{2}t}D^{-}_{W,\sigma,\mu,n}f(p)\\
	&+\frac{1}{\nu}e^{-\frac{1+\nu}{2}t}\int_{t}^{\infty}e^{-\frac{1-\nu}{2}\xi}G(\xi)\;\text{d}\xi-\frac{1}{\nu}e^{-\frac{1-\nu}{2}t}\int_{t}^{\infty}e^{-\frac{1+\nu}{2}\xi}G(\xi)\;\text{d}\xi\;,
	\end{split}
\end{equation*}
where we define the coefficients $D^{\pm}_{W,\sigma,\mu,n}f\colon Y\to \C$ by setting
\begin{equation*}
	D^{+}_{W,\sigma,\mu,n}f(p)=c_2-\frac{1}{\nu}\int_{t_1}^{\infty}e^{-\frac{1-\nu}{2}\xi}G(\xi)\;\text{d}\xi\;, \quad 	D^{-}_{W,\sigma,\mu,n}f(p)=c_1+\frac{1}{\nu}\int_{t_1}^{\infty}e^{-\frac{1+\nu}{2}\xi}G(\xi)\;\text{d}\xi
\end{equation*}
for every\footnote{Observe that the dependence on the point $p\in Y$ lies, implicitly in the adopted notation, in the function $G(\xi)$ appearing in the integrand, which is defined in~\eqref{eq:G}.} $p \in Y$.

\pagebreak

Now notice that, for every $p \in Y$, we may estimate

\begin{equation*}
	\begin{split}
	|D^{\pm}_{W,\sigma,\mu,n}f(p)|&\leq \sup\{|c_1|,|c_2|  \} +\frac{1}{|\nu|}\int_{t_1}^{\infty}e^{-\frac{1\mp \nu}{2}\xi}|G(\xi)|\;\text{d}\xi\\
	&\leq \sup\{|c_1|,|c_2| \}+\frac{1}{|\nu|}\sup_{\xi\geq t_1}|G(\xi)|\int_0^{\infty}e^{-\frac{1-\nu}{2}\xi}\text{d}\xi\\
	&\leq \frac{\kappa(W,\mu,n)}{\sigma}\norm{f}_{\mathscr{C}^{1}}
	\end{split}
\end{equation*}
for a certain constant $\kappa(W,\mu,n)>0$, as emerges from the explicit expression in~\eqref{eq:G} for the function $G$ and by carrying out straightforward computations to determine the coefficients $c_1,c_2$. Similarly, it is easy to verify that the  absolute value of the remainder term
\begin{equation*}
	\mathcal{R}_{W,\sigma,\mu,n}f(p,t)= \frac{1}{\nu}e^{-\frac{1+\nu}{2}t}\int_{t}^{\infty}e^{-\frac{1-\nu}{2}\xi}G(\xi)\;\text{d}\xi-\frac{1}{\nu}e^{-\frac{1-\nu}{2}t}\int_{t}^{\infty}e^{-\frac{1+\nu}{2}\xi}G(\xi)\;\text{d}\xi
\end{equation*}
can be bounded from above, up to a constant depending on $W,\mu,n,\norm{f}_{\mathscr{C}^{1}}$ and $\sigma$, by the lower-order term $e^{-t}$.

This finalizes the argument.

\section{Further directions and open questions}
\label{sec:further}

The scope of applicability of the method employed in Section~\ref{sec:proof} to derive the asymptotics in Theorem~\ref{thm:expandingtranslates} appears to be much broader. The extension to expanding homogeneous submanifols inside unit tangent bundles of compact hyperbolic manifolds of any dimension, beyond the already known cases (cf.~the works of Södergren~\cite{Sodergren} and Lutsko~\cite{Lutsko}, for example), constitutes a first prospect for future research. The question would then arise naturally as to which extent effective equidistribution results of this sort can be generalized to arbitrary submanifolds, for instance general rectifiable arcs in our two-dimensional setting (cf.~\cite{Bufetov-Forni}).

The abstract formulation of the equidistribution problem framed in the introduction does not demand compactness of the manifold $M$, but only finiteness of its volume\footnote{In fact, even volume finiteness is immaterial for the more general version stated in Question~\ref{qst:initial}.}. While this distinction does not concern the Euclidean case, as the quotient of $\R^d$ by any lattice in\linebreak $\Or_d(\R)\ltimes \R^d$ is compact, our assumption of compactness places a significant restriction in the negative-curvature setting. For a glimpse of the intrinsic challenges lurking beneath any extension of our arguments and results in the direction of finite-volume non-compact hyperbolic surfaces, which from a technical perspective would entail dealing with the considerably more involved spectral theory of the Casimir operator, we draw here the reader's attention to the close linkage, pinned down by Zagier in~\cite{Zagier}, between attaining the optimal equidistribution rate for expanding closed horocycles on the unit tangent bundle of the modular surface $\SL_2(\Z)\bsl \Hyp$ and proving the Riemann hypothesis. Motivated by such work, Sarnak established in~\cite{Sarnak} an effective rate of equidistribution for geodesic translates of closed orbits of the horocycle flow in the finite-volume non-compact case. 

Equally fraught with difficulties, stemming this time from the absence of an algebraic description for the ambient spaces and the flows acting on them (a description which fundamentally underpins our approach), is the case where the ambient manifold has non-positive mixed curvature. For compact manifolds of non-constant negative curvature, an approach based upon the theory of transfer operators might prove fruitful: in this regard, we mention the recent work of Adam and Baladi~\cite{Adam-Baladi} on effective decay rates for horocycle ergodic averages.

There is an eminently vast body of literature pertaining to equidistribution problems of a piece with those examined in this manuscript. Without striving for completeness, we list here below an array of contributions, arranged by topic, with which this section draws to a close.

\pagebreak

\begin{itemize}
	\item For the asymptotic distribution of large circles on flat surfaces, see the works of Chaika-Hubert~\cite{Chaika-Hubert} and Colognese-Pollicott~\cite{Colognese-Pollicott}.
	\item Equidistribution of dilated curves on nilmanifolds has been treated in works of Kra-Shah-Sun~\cite{Kra-Shah-Sun} and Björklund-Fish~\cite{Bjorklund-Fish}.
	\item The investigation of equidistribution of expanding translates of curves on other homogeneous spaces is an extremely active line of research, for which a constant impetus is given by the wide-ranging implications in questions of Diophantine approximation. This direction has been extensively pursued by Shah in~\cite{Shah,Shah-second,Shah-secondprime,Shah-third,Shah-fourth} and, subsequently, by L.~Yang in~\cite{Yang,Yang-second, Yang-third}, Khalil in~\cite{Khalil} and P.~Yang in~\cite{PYang}. 
	\item Concerning equidistribution of expanding horospheres, we mention, together with the already cited articles of Zagier~\cite{Zagier}, Sarnak~\cite{Sarnak}, Strömbergsson~\cite{Strombergsson-closedhorocycles}, Södergren~\cite{Sodergren}, Lutsko~\cite{Lutsko} and Edwards~\cite{Edwards}, the works of Hejhal~\cite{Hejhal}, Flaminio-Forni~\cite{Flaminio-Forni},\linebreak Kleinbock-Margulis~\cite{Kleinbock-Margulis},  Kleinbock-Weiss~\cite{Kleinbock-Weiss} and Mohammadi-Oh~\cite{Mohammadi-Oh}.
	\item A vast theme is represented by the quest for the asymptotic distribution properties of expanding translates of orbits of symmetric subgroups inside homogeneous spaces of semisimple Lie (or algebraic) groups. This was sparked by the groundbreaking contributions contained in Margulis' thesis~\cite{Margulis-thesis}, relating various lattice-point counting problems\footnote{We will not touch upon such connections in any detail here, referring instead, by way of example, to our earlier work~\cite{Corso-Ravotti} for the case of the Gauss circle problem in the hyperbolic plane.
	
	We add an historical note here, referring to~\cite[Chap.~VIII, Sec.~5.12]{Stein} for the relevant references: problems revolving around the distribution of lattice points in expanding regions of Euclidean spaces also served as the first catalytic agent for the study of Fourier decay of surface measures, which entered the picture decisively in our treatment of equidistribution of expanding sets on tori (Sections~\ref{sec:euclidean} and~\ref{sec:quantitativeeuclidean}).} to such equidistribution features. Margulis' insights were later developed in a general framework in the seminal articles of Duke-Rudnick-Sarnak~\cite{Duke-Rudnick-Sarnak} and Eskin-McMullen~\cite{Eskin-McMullen}, and have propelled a great deal of research ever since; see, to name but a few, the works of Eskin-Margulis-Mozes~\cite{Eskin-Margulis-Mozes}, Benoist-Oh~\cite{Benoist-Oh} and Oh-Shah~\cite{Oh-Shah}. 
	\item An example of equidistribution of expanding spheres in a discrete setup, specifically in the quotient of a tree by a lattice of tree automorphisms, is to be found in\linebreak Ciobotaru-Finkelshtein-Sert~\cite{Ciobotaru-Sert}.
	
\end{itemize}

\appendix

\section{Effective equidistribution of expanding circles on the unit tangent bundle of the $2$-torus}
\label{secapp:app}
	
	The purpose of this appendix is to establish an effective equidistribution result for expanding circles in the unit tangent bundle $T^1\T^2$ of the two-dimensional torus. This complements the analogous statement for unit tangent bundles of compact hyperbolic surfaces, which is Theorem~\ref{thm:expandingtranslates} specialized to $W=\Theta$ (see Section~\ref{sec:unittangentbundles} for the notation).
	
	The restriction to dimension $d=2$ allows for a treatment which fully parallels the arguments in Sections~\ref{sec:euclidean} and~\ref{sec:quantitativeeuclidean}, as $T^1\T^2$ is a trivial circle bundle over $\T^2$, and as such can be canonically identified with $\T^3$. For every $t>0$, the projection to $\T^2$ of the circle of radius $t$ centered at the origin in $\R^2$ admits the parametrization
	\begin{equation*}
		[0,1]\ni u \mapsto (t\cos{2\pi u},t\sin{2\pi u})+\Z^2\; \in \T^2\;.
	\end{equation*}

The canonical lift $\overline{C}_t$ to $T^1\T^2$, constructed in the same manner as in the hyperbolic setting of Section~\ref{sec:unittangentbundles}, is easily seen to be parametrized by
\begin{equation*}
	[0,1]\ni u\mapsto (t\cos{2\pi u},t\sin{2\pi u},u)+\Z^3 \;\in \T^3,
\end{equation*}
under the above mentioned identification of $T^1\T^2$ with $\T^3$. If $f\colon \T^3\to \C$ is a continuous function with summable Fourier transform, then the average of $f$ with respect to the uniform measure $\overline{m}_t$ on $\overline{C}_t$ is given by
\begin{equation*}
	\sum_{N\in \Z^3}\hat{f}(N)\int_{0}^{1}e^{2\pi iN\cdot (t\cos{2\pi  u}, t\sin{2\pi u},u)}\text{d}u\;;
\end{equation*}
hence, the discrepancy with the uniform average of $f$ over $\T^3$ equals
\begin{equation}
	\label{eq:discrepancy}
	\begin{split}
	\sum_{N\in \Z^{3}}\hat{f}(N)&\int_{0}^{1}e^{2\pi iN\cdot (t\cos{2\pi  u}, t\sin{2\pi u},u)}\text{d}u-\int_{\T^3}f\;\text{d}m_{\T^3}\\
	&=\sum_{N=(N_1,N_2,N_3)\in \Z^{3}\setminus \{0\}}\hat{f}(N) \hat{\nu}(-tN_1,-tN_2,-N_3)
	\end{split}
\end{equation}
where $\nu$ is the uniform probability measure on the curve
\begin{equation*}
 \gamma\colon [0,1]\to \R^3\;, \quad  u\mapsto (\cos{2\pi u},\sin{2\pi u},u)\;.
 \end{equation*}
 The last sum can actually be restricted to those $N=(N_1,N_2,N_3)\in \Z^3$ with\linebreak $(N_1,N_2)\neq (0,0)$, as for the others the corresponding summand vanishes by orthogonality of characters: $\int_{0}^1e^{2\pi i N_3u}= 0$ for every $N_3\neq 0$. 
 Taking absolute values, the expression we have thus obtained in~\eqref{eq:discrepancy} is majorized by
\begin{equation*}
	\lnorm{\hat{f}}_{\ell^{1}(\Z^3)}\;\sup_{N=(N_1,N_2,N_3)\in \Z^{3},\; (N_1,N_2)\neq (0,0)}|\hat{\nu}(-tN_1,-tN_2,-N_3)|\;.
\end{equation*}
Recall that the contact type of the curve $\gamma$ at a point $u_0\in (0,1)$, with respect to affine hyperplanes in $\R^3$, is defined as the infimum of all integers $k\geq 1$ for which the following holds: for every unit vector $\eta\in \R^3$, there exists $j\leq k$ such that the $j$-th derivative of the real-valued function $u\mapsto (\gamma(u)-\gamma(u_0))\cdot \eta$ at $u_0$ is non-zero. The supremum of such quantities over all points $u_0\in (0,1)$ is called the contact type of $\gamma$ inside $(0,1)$ (cf.~\cite[Chap.~VIII]{Stein}). A moment's computation allows to check that $\gamma$ has contact type $2$, whence  ~\cite[Chap.~VIII, Thm.~2]{Stein} delivers $|\hat{\nu}(\xi)|\leq C(|\xi|^{-1/2})$ for every $\xi \in \R^3$ and for some constant $C=C(\nu)>0$. From this, we readily deduce that  
\begin{equation*}
\begin{split}
\biggl|\int_{\T^3}f\;\text{d}\overline{m}_t-\int_{\T^3}f\;\text{d}m_{\T^3}\biggr|&\leq C\; \lnorm{\hat{f}}_{\ell^1(\Z^3)}\sup_{N=(N_1,N_2,N_3)\in \Z^{3},\; (N_1,N_2)\neq (0,0)}|(tN_1,tN_2,N_3)|^{-1/2}\\
&=C\;\lnorm{\hat{f}}_{\ell^{1}(\Z^3)}\;t^{-1/2}\;.
\end{split}
\end{equation*}

We gather the upshot of the analysis above in the following proposition.

\begin{prop}
Let $C_t\subset \T^2$ be the canonical projection of the circle of radius $t$ centered at the origin in $\R^2$, and let $\overline{C}_t$ be its lift to the unit tangent bundle $T^1\T^2$ obtained by attaching, to each point $x$ of $C_t$, the outward-pointing normal vector to $C_t$ at $x$. Identifying $T^1\T^2$ with $\T^3$ in the canonical fashion, and letting $\overline{m}_t$ be the uniform probability measure on $\overline{C}_t$, the following holds: there exists $C>0$ such that, for every continuous function $f\colon \T^3\to \C$ with summable Fourier transform and every $t>0$,  
\begin{equation*}
	\biggl|\int_{\T^3}f\;\emph{d}\overline{m}_t-\int_{\T^3}f\;\emph{d}m_{\T^3}\biggr|\leq C\;\lnorm{\hat{f}}_{\ell^{1}(\Z^3)}\;t^{-1/2}\;.
\end{equation*}
\end{prop}

We infer that expanding circles on $\T^2$, when lifted to the unit tangent bundle, equidistribute effectively towards the uniform measure according to the same rate governing their effective equidistribution on the base $\T^2$ (cf.~Corollary~\ref{cor:dilatingsurfaces}).  We expect this to be the case for expanding spheres and their lifts in all dimensions $d$, though for $d>2$ the classical Fourier-analytic approach ceases to work on the level of unit tangent bundles, and harmonic analysis on products of tori and spheres needs to be invoked instead. 

It is worth highlighting that the case $d=2$ examined in this appendix is intimately tied to the celebrated Gauss circle problem in number theory, asking for a precise asymptotics for the number of integer points in disks of increasing radius in $\R^2$; the connection is touched upon in Section~\ref{sec:further}, to which we refer for the relevant literature. It seems rather preposterous, however, to presume that such a connection has the potential to lead to any improvement on the currently known sharpest estimate on the error. An exhaustive account of the historical developments around Gauss' circle problem can be found in~\cite{Ivic}.

\end{document}